\documentclass[11pt, two side]{article}
\usepackage{amsfonts}

\usepackage{latexsym,amsmath,amsfonts,mathrsfs,wasysym,times,lineno,amscd, mathtools,mathrsfs}
\usepackage{amsthm,amscd,amsfonts,newlfont}
\usepackage{amssymb, upref, color}
\usepackage{pst-all}
\usepackage{epsf, graphicx, subfigure, verbatim, epic}
\usepackage{latexsym}
\usepackage[colorlinks]{hyperref}
\usepackage{mathrsfs}
\usepackage{array}
\usepackage{arydshln}
\usepackage{soul}

\numberwithin{equation}{section}

\theoremstyle{plain}
\newtheorem{exam}{Example}[section]
\newtheorem{theorem}[exam]{Theorem}
\newtheorem{lemma}[exam]{Lemma}
\newtheorem{remark}[exam]{Remark}

\newtheorem{definition}[exam]{Definition}
\newtheorem{corollary}[exam]{Corollary}

\newtheorem{conjecture}[exam]{Conjecture}

 \def\R{\mathbb R}
 \def\N{\mathbb N}

\title{Spiderweb central configurations}
\author{Olivier H\'enot and
Christiane Rousseau, Universit\'e de Montr\'eal\thanks{This research was supported by NSERC in Canada.}}

\begin{document}

\date{}

\maketitle

\begin{abstract} 
In this paper we study spiderweb central configurations for the $N$-body problem, i.e configurations given by $N=n \times \ell+1$ masses located at the intersection points of $\ell$ concurrent equidistributed half-lines with $n$ circles  and a central mass $m_0$,
under the hypothesis that the $\ell$ masses on the $i$-th circle are equal to a positive constant $m_i$;
we allow the particular case $m_0=0$.
We focus on constructive proofs of the existence of spiderweb central configurations, which allow numerical implementation. Additionally, we prove
the uniqueness of such central configurations when $\ell \in \{2,\dots,9\}$
and arbitrary $n$ and $m_i$; under the constraint $m_1\geq m_2\geq \ldots \geq m_n$ we also prove uniqueness for $\ell \in \{10,\dots,18\}$ and $n$ not too large. We also give an algorithm providing a rigorous proof of the existence and local unicity
of such central configurations when given as input a choice of $n$, $\ell$ and $m_0, . . . ,m_n$. Finally, our numerical simulations highlight some interesting properties of the mass distribution.
\end{abstract}

\section{Introduction}

The $N$-body problem consists in describing the positions $\mathbf{r}_1(t),\dots,\mathbf{r}_N(t)$ of $N$ masses $m_1,\dots,m_N$ interacting through Newton's gravitational law, which are solutions of the  system of coupled non-linear equations

\begin{equation}\label{eq_newton}
m_i \ddot{\mathbf{r}}_i =  -\sum_{j\neq i} \frac{G m_i m_j(\mathbf{r}_i - \mathbf{r}_j)}{| \mathbf{r}_i - \mathbf{r}_j |^3}=-\frac{ \partial U (\mathbf{r})}{\partial \mathbf{r}_i} = \mathbf{F}_i(\mathbf{r})  \,,\qquad U = -\sum_{i<j} \frac{G m_i m_j}{|\mathbf{r}_i-\mathbf{r}_j|},\end{equation}
\noindent for $i=1,\dots,N$, with $\mathbf{r}=(\mathbf{r}_1,\dots,\mathbf{r}_N)\in \R^{3N}_{>0}$,  where $G$ denotes the gravitational constant.

Specific solutions, called \emph{central configurations}, arise when the acceleration of each mass-particle is proportional to the position with the same constant of proportionality (depending on time) for all masses. In this paper we are interested in \emph{spiderweb configurations} of $N=n \times \ell +1$ masses, where the masses are located at the intersection points of $\ell$ concurrent equidistributed half-lines  with $n$ circles of radii $r_1< \dots <r_n$, and a central mass $m_0$, under the hypothesis that the $n$ masses on the $i$-th circle are equal to a positive constant $m_i$, while the mass $m_0$ is allowed to vanish (see Figure~\ref{fig:intro}). 

\begin{figure}
\centering
\subfigure[$n=3$, $\ell=5$]{
\label{fig:a}
\includegraphics[height=2in]{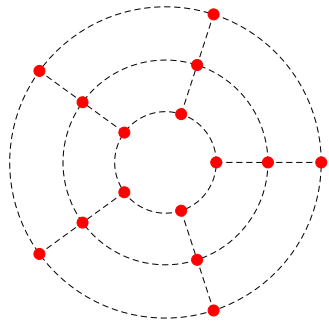}}
\hspace{8pt}
\subfigure[$n=4$, $\ell=6$]{
\label{fig:b}
\includegraphics[height=2in]{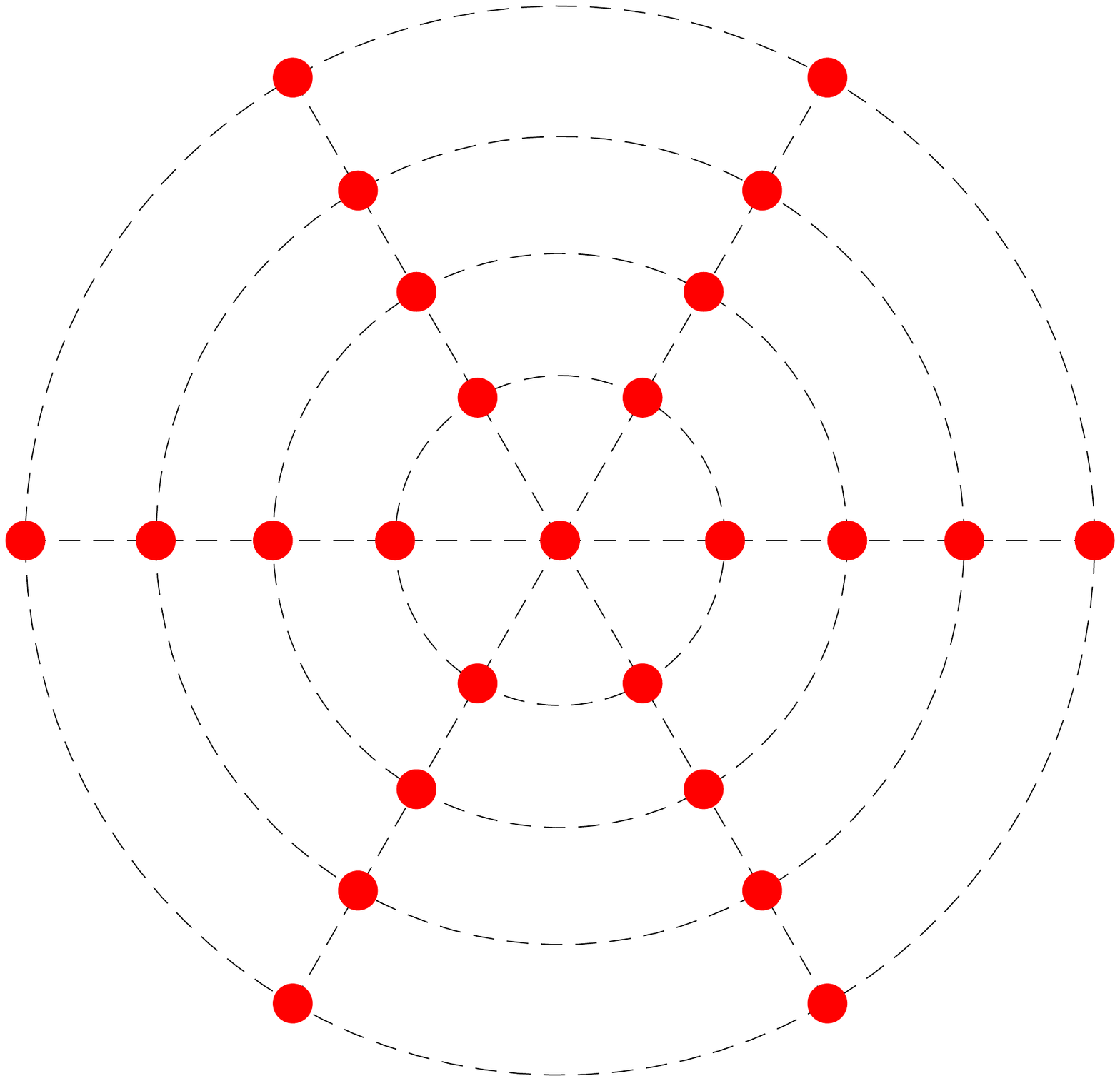}}\\
\subfigure[$n=8$, $\ell=16$]{
\label{fig:c}
\includegraphics[height=2in]{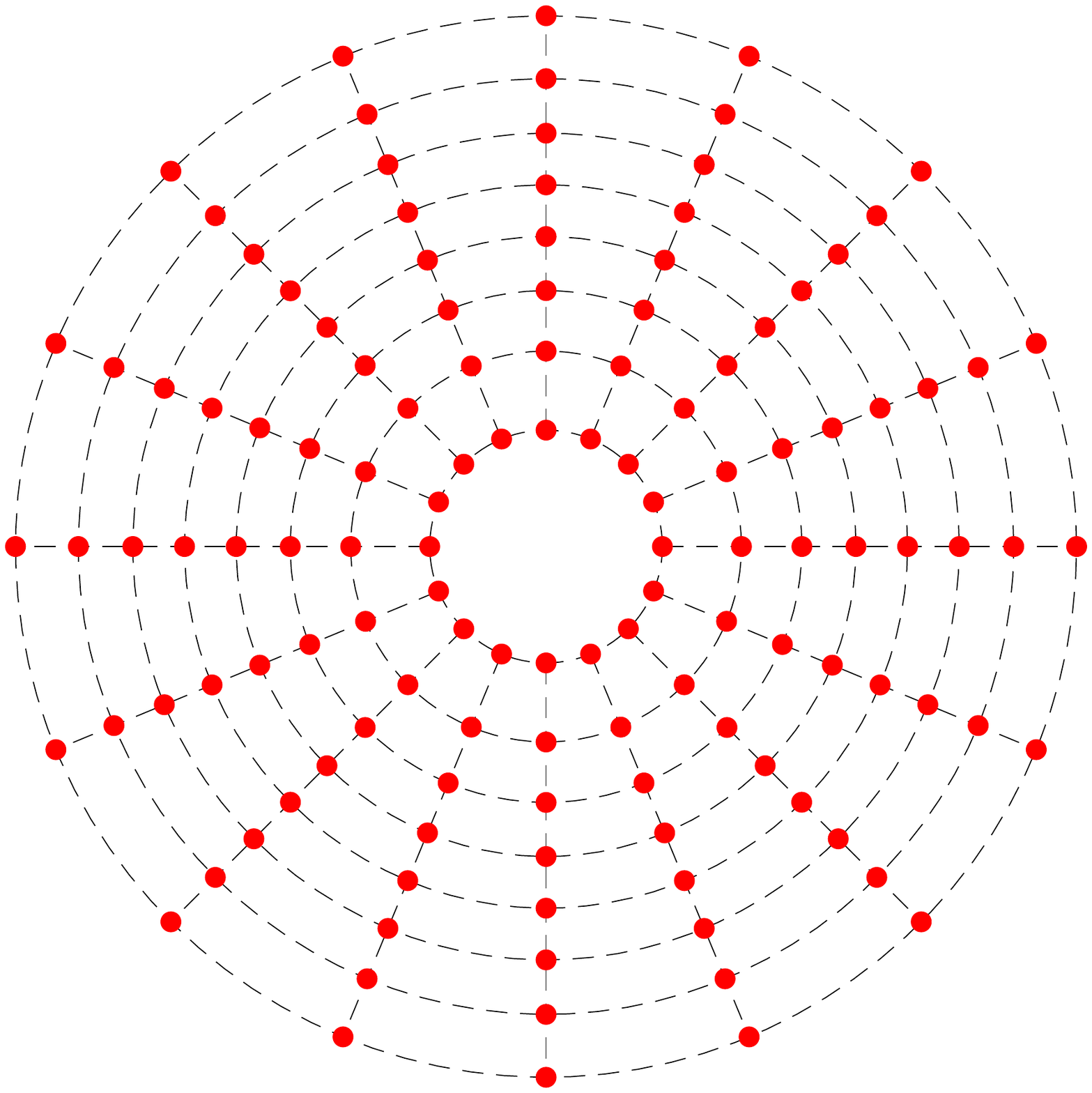}} 
\hspace{8pt}
\subfigure[$n=3$, $\ell=5$]{
\label{fig:d}
\includegraphics[height=2in]{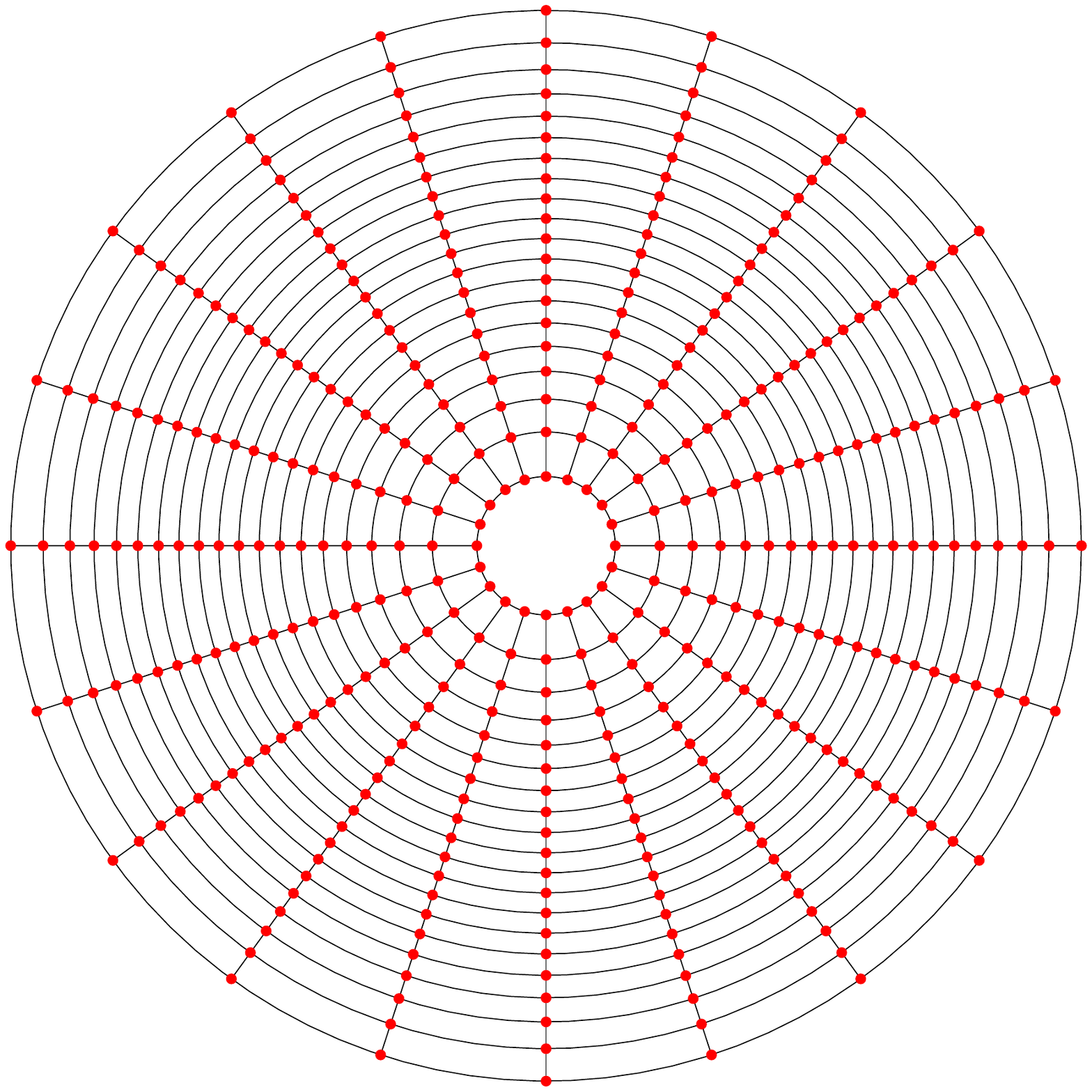}}
\caption{Spider central configurations with unitary masses. Note that $m_0\neq0$ only in (c).}
\label{fig:intro}
\end{figure}

The existence of such central configurations have been studied in the literature, often  in the special case $m_0=0$, starting with Moulton in 1910 (\cite{Mou}), which settled the case $\ell=2$ as a particular case of $N$ aligned masses. 
The case $n=1$ has been treated by Maxwell in the 19th century \cite{Maxwell}. Later Moeckel and Simo \cite{MS} proved the existence and uniqueness of such a configuration in the case $n=2$ and $m_0=0$.
Corbera, Delgado and Llibre \cite{CDL} considered the case of $n$ and $\ell$ arbitrary with restrictions on the masses  of the type $m_1\gg \ldots \gg m_n$ for $m_0=0$, while Saari treated the general case, releasing the restrictions on the masses with a different method in \cite{Sa1} and \cite{Sa2}. We became interested in the problem and decided to study numerically the distribution $M(\eta)$ of mass depending of the distance $\eta$ to the origin for large values of $\ell$ and $n$. At the same time we considered the proofs of the results appearing in \cite{CDL}, \cite{Sa1} and \cite{Sa2}: to our surprise, these proofs are incomplete and we started working on completing them. We could give some complete proofs, but not for all values of $n$, $\ell$ and of the masses. But our proofs are constructive and can be implemented numerically. Our numerical experimentations suggest the uniqueness of the central configurations (as claimed by Saari), and allow exploring the mass distribution in these configurations for large values of $n$ and $\ell$. By the time a first version of this paper was ready, we learnt of the general result of Montaldi \cite{Montaldi}, giving the existence of central configurations with a symmetrical mass distribution \cite{Montaldi}. The proof of Montaldi, based on a variational formulation of the problem and using the principle of symmetric criticality of Palais, is very elegant. However, his proof is completely existential. Hence we believe that our proofs are a complement to the one provided by Montaldi in \cite{Montaldi}. Additionally, for $\ell \in \{2,\dots,9\}$ and a few other particular cases, we could prove the uniqueness of the central configurations.

The proof of Corbera, Delgado and Llibre is by induction on $n$. To go from $n$ circles to $n+1$ circles, the idea is to add an $n+1$-th circle with masses $m_{n+1}=0$ and to allow the mass $m_{n+1}$ to increase, via the implicit function theorem. The use of the implicit function theorem requires some invertible Jacobian. Since the authors could not prove that the Jacobian is invertible, they use an uniqueness argument to claim that the unique solution can be extended for nonzero values of $m_{n+1}$. The argument is not valid as shown by the following counter-example $f(x,y,m) = (x^2+y^2+m^2,y+m)$, which has a unique zero for $m=0$ and no zero for $m\neq0$. However the proof can easily be repaired and we include in the paper a proof of the existence of such configurations for $m_1\gg \ldots \gg m_n$, which is much shorter than the one in \cite{CDL}.

By adapting the method of \cite{CDL} in the spirit of Moulton \cite{Mou}, i.e. starting from a restricted $N$-body problem and following the solutions by the implicit function theorem for large values of the $m_i$, we were able to prove the existence and uniqueness of spiderweb central configurations for $\ell \in \{2,\dots,9\}$
and arbitrary $n$ and $m_i$. Under constraints on the mass distribution and the maximum number of circles, the uniqueness is also proven for $\ell \in \{10,\dots,18\}$.

In \cite{Sa2}, Saari proposed a proof of the existence of spiderweb central configurations in the general case. There, again, the proof was by induction on $n$, and used continuity arguments, which had to rely on the implicit function theorem. But no checking of the hypothesis of the implicit function theorem could be found. Our checking of these hypotheses revealed much harder than expected, but we could adapt the method of Saari and prove the existence of spiderweb central configurations for arbitrary $\ell$ and $m_i$ and $n\in\{3,4\}$.

To conclude, we give an algorithm providing a rigorous proof of the existence and local unicity
of such central configurations when given as input a choice of $n$, $\ell$ and $m_0, . . . ,m_n$. The algorithm
has been applied to all $n\leq100$ and all even values $\ell\leq 200$ when $m_0=0$ and all masses are equal. We have also applied it in the case of different masses.
Our numerical explorations allowed us studying the profile of the function $M(\eta)$ describing the distribution of mass at the distance $\eta$ from the center of mass. This profile reveals universal features that are quite interesting.

The paper is organized as follows. Section~\ref{sec:prel} contains preliminaries. Section~\ref{sec:general} shows the existence of spiderweb central configurations with $N=n\times \ell+1$ or $N= n\times\ell$, and arbitrary $n$ and $\ell$. In Section~\ref{sec:analy2} we prove the existence and uniqueness of spiderweb central configurations for $\ell \in \{2,\dots,9\}$,
and arbitrary $n$ and $m_i$ in the spirit of \cite{Mou}, while in Section~\ref{sec:analy1} we give a constructive proof of the existence of spiderweb configurations for $n\in\{3,4\}$  and arbitrary $\ell$.
Finally Section~\ref{sec:numer} deals with the numerical algorithm providing rigorous proof of existence, while Section~\ref{sec:mass} studies the properties of  the function $M(\eta)$. 

\section{Preliminaries}\label{sec:prel}

\subsection{Scalings and central configurations in \eqref{eq_newton}}\label{sec:scalings}

For simplicity, we translate the center of mass at the origin. Considering changes $(r,m,t)\mapsto (Ar,Bm,Ct)$ in the units of length, mass and time 
satisfying $A^3B^{-1}C^{-2}=G$ scales $G=1$. There remains two degrees of freedom: indeed additional changes preserve $G=1$ provided $A^3=BC^2$.

\begin{definition}\label{def:CC}
The configuration of $N$ bodies is \emph{central} at some time $t^*$ if $\ddot{\mathbf{r}}(t^*) = \lambda \mathbf{r}(t^*)$ for some common $\lambda$, where $\mathbf{r}=(\mathbf{r}_1,\dots,\mathbf{r}_N)\in \R^{3N}_{>0}$.
\end{definition}

\begin{remark}
The previous definition suggests that being a central configuration is a characteristic of the precise time $t^*$. However, it is well-known that, for well-chosen initial velocities, the $N$ bodies remain in a central configuration for all time $t$; during the motion of the $N$ bodies, the common $\lambda$ is a function of $t$.
\end{remark}

It is easy to see that $\lambda$ is a strictly negative value given by $\lambda=U/I<0$ where $I = \sum m_i \mathbf{r}_i^2$ is the moment of inertia. A scaling in time allows to take $\lambda=-1$. Keeping $\lambda=-1$, 
if $m \in \R_{\geq 0}^N$ is the vector of masses, then if $(\mathbf{r},m)$ is a central configuration, so is $(A\mathbf{r},A^3 m)$.

Moreover, by the definition of a central configuration and the equation of motion \eqref{eq_newton}, any homothety and rotation of the positions, i.e. $\mathbf{r}=(\mathbf{r}_1,\dots,\mathbf{r}_N) \mapsto A\,\Omega \mathbf{r}$ where $A>0$ and $\Omega \in SO(3)$, yields a central configuration with $\lambda/\gamma^3$. Hence, when discussing the uniqueness of central configurations we mean the uniqueness of the equivalence class for the equivalence relation $(\mathbf{r},m, \lambda )\equiv \left(A\,\Omega\mathbf{r}, Bm, \frac{\lambda}{A^3C^2}\right)$ under the condition $A^3=BC^2$ which we can also write

\begin{equation}
(\mathbf{r},m, \lambda )\equiv \left(A\,\Omega\mathbf{r}, \frac {A^3m}{C^2}, \frac{\lambda}{A^3C^2}\right).\label{equiv_relation}
\end{equation}

\subsection{Spiderweb configurations}

We consider \textit{spiderweb central configurations} formed by $n\times\ell$ masses located at the intersection points of $n$ circles centered at the origin of radii $r_1,\dots,r_n$, with $\ell$ half-lines starting at the origin,
whose angle with the positive $x$-axis is $\theta_k =2 \pi k / \ell$ for $k=0,\dots,\ell-1$, together with a mass $m_0$ placed at the origin, under the hypothesis that the $\ell$ masses on the $i$-th circle are equal to a positive constant $m_i$.

By symmetry, it is clear that the gravitational tug $\mathbf{F}_0$ on the mass $m_0$ located at the origin is identically zero, and  thus $\ddot{\mathbf{r}}_0=\lambda \mathbf{r}_0$ for any $\lambda$.

Let

\[
\mathcal{R}^n = \{ r \in \R^n \, : \, 0 < r_1 < \ldots < r_n\}.
\]

\noindent Rearranging the terms and using the symmetry in the equation \eqref{eq_newton} so that there are $n$ bodies on the positive horizontal axis, it suffices to consider the following system

\begin{align}\label{eq_spiderweb}
    \ddot{r}_i = -\sum_{k=1}^{\ell-1} \frac{m_i }{2^{3/2}r_i^2( 1 - \cos \theta_k )^{1/2}} -\frac{ m_0 }{r_i^2} - \sum_{\begin{smallmatrix}j=1\\j\neq i \end{smallmatrix}}^n \sum_{k=0}^{\ell-1} \frac{m_j( r_i - r_j \cos \theta_k )}{( r_i^2 + r_j^2 - 2 r_i r_j \cos \theta_k )^{3/2}} ,
\end{align}

\noindent for $i=1,\dots,n$, with $\theta_k = \frac{2\pi k}{\ell}$ and $r = (r_1,\dots,r_n)\in \mathcal{R}^n$.

\subsection{Tools}

Under the previous reduction to \eqref{eq_spiderweb}, the contribution of the gravitational force on the mass located at $(r_i,0)$ is $F_i(r) = \sum_{j=1}^n F_{ij}(r_i,r_j)$, where $F_{ij}( r_i,r_j)$ is the contribution of $F_i(r)$ coming from the attraction of the $j$-{th} circle, given by

\[
\frac{F_{ij}(r_i,r_j)}{m_i} =
\begin{cases}
\displaystyle -\frac{ m_0 }{r_i^2}, & j=0,\\
\displaystyle -\sum_{k=1}^{\ell-1} \frac{m_i }{2^{3/2}r_i^2( 1 - \cos \theta_k )^{1/2}}, & j=i, \\
\displaystyle -\sum_{k=0}^{\ell-1} \frac{m_j(r_i-r_j \cos \theta_k)}{( r_i^2 + r_j^2 - 2 r_i r_j \cos \theta_k )^{3/2}}, & j\neq i \text{ and } j>0,
\end{cases}
\]

We introduce

\[
\zeta_\ell=\sum_{k=1}^{\ell-1} \frac{1}{( 1 - \cos \theta_k )^{1/2}}\,, \] and 
\begin{equation}
\begin{cases}
x_j=r_i/r_j, & \text{si } j>i,\\
y_j=r_j/r_i, & \text{si } j<i,\\
\end{cases}
\label{def:x}\end{equation}

\noindent so that $F_{ij}(r_i,r_j)/m_i$ becomes

\[
\frac{F_{ij}}{m_i} =
\begin{cases}
\displaystyle -\frac{ m_0 }{r_i^2} , & j=0,\\
\displaystyle -\frac{m_i}{2^{3/2}r_i^2}\zeta_\ell, & j=i, \\
\displaystyle -\frac{m_j}{r_i^2}\sum_{k=0}^{\ell-1} \frac{1-y_j\cos\theta_k}{(1+y_j^2-2y_j\cos \theta_k)^{3/2}}, & 0<j<i, \\
\displaystyle -\frac{m_j x_j^2}{r_i^2}\sum_{k=0}^{\ell-1} \frac{x_j-\cos\theta_k}{(1+x_j^2-2x_j\cos \theta_k)^{3/2}}, & j >i.
\end{cases}
\]

Let 
\begin{equation}\phi_\nu(x) = \sum_{k=0}^{\ell-1}\frac{1}{d_k^\nu(x)},\qquad d_k(x)= (1+x^2-2x\cos\theta_k)^{1/2},\label{def:phi}\end{equation} so that

\begin{align*}
\sum_{k=0}^{\ell-1} \frac{x_j-\cos\theta_k}{(1+x_j^2-2x_j\cos \theta_k)^{3/2}}
&=-\phi'_1(x_j),\\
\sum_{k=0}^{\ell-1} \frac{1-y_j\cos\theta_k}{(1+y_j^2-2y_j\cos \theta_k)^{3/2}}
&= \sum_{k=0}^{\ell-1} \frac{(1+y_j^2-2y_j\cos \theta_k)-y_j^2+y_j \cos \theta_k}{(1+y_j^2-2y_j\cos \theta_k)^{3/2}}\\
&= \phi_1(y_j) + y_j\phi_1'(y_j).
\end{align*}

Whence,

\begin{equation}\label{Fij}
\frac{F_{ij}}{m_i} =
\begin{cases}
\displaystyle -\frac{ m_0 }{r_i^2}\,, & j=0,\\
\displaystyle -\frac{ m_i }{2^{3/2}r_i^2}\zeta_\ell, & j=i, \\
\displaystyle -\frac{m_j}{r_i^2}\left(\phi_1(y_j)+y_j\phi_1'(y_j)\right), & 0<j<i, \\
\displaystyle \frac{m_j x_j^2}{r_i^2} \phi_1'(x_j), & j >i,
\end{cases} \qquad \text{pour } i=1,\dots,n.
\end{equation}

\begin{lemma}\cite{MS}\label{lem:MS}
Let $\phi_\nu(x) = \sum_{k=0}^{\ell-1} 1 / d_k^\nu(x)$ with $d_k(x)= (1+x^2-2x\cos\theta_k)^{1/2}$ and $\nu>0$. Then, for $-1<x<1$, $\phi_\nu(x)$ is analytic and all the coefficients of its power series expansion are positive. In particular, for $0<x<1$, $\phi_1(x)$ is analytic and all its derivative are positive.
\end{lemma}

\begin{lemma}\label{lem:dico}
Let $F_{ij}$ defined by \eqref{Fij} and

\[
\lambda_{ij} = \frac{F_{ij}}{m_i r_i},
\]

\noindent such that $\lambda_i=\sum_{j=1}^n\lambda_{ij}$ for $i=1,\dots,n$. We have the four following properties:
\begin{enumerate}
\item
\[
F_{ij}
\begin{cases}
>0, &i<j,\\
<0, &i\geq j,
\end{cases}
\qquad \text{and} \qquad
\lambda_{ij}
\begin{cases}
>0, &i<j,\\
<0, &i\geq j.
\end{cases}
\]
\item $\partial_{r_j} \lambda_i<0$ for all $i\neq j$. 
\item $\partial_{r_i} \lambda_i >0$ for all $i$.
\item $0>\partial_{r_k} \lambda_{ik}>\partial_{r_k} \lambda_{j k}$ for all $i<j<k$.
\item $\partial_{r_k} \lambda_{j k}<\partial_{r_k} \lambda_{ik}<0$ for all $0<k<j<i$.
\end{enumerate}
\end{lemma}

\begin{proof}
\begin{enumerate}
\item Direct consequence of Lemma~\ref{lem:MS}.

\item We have

\[
\lambda_{ij}=
\begin{cases}
-\frac{m_0}{r_i^3}, & j=0,\\
- \frac{m_j}{r_i^3}(\phi_1(y_j)+y\phi_1'(y_j)), & 0<j<i ,\\
\frac{m_j}{r_i^3}x_j^2\phi_1'(x_j), & j>i,
\end{cases}
\]

\noindent and the chain rule give

\[
\partial_{r_j} \lambda_{ij}=
\begin{cases}
0, & j=0,\\
\frac{\partial y_j}{\partial r_j}\frac{\partial \lambda_{ij}}{\partial y_j} = \frac{1}{r_i}\frac{\partial \lambda_{ij}}{\partial y_j} < 0, & 0<j<i ,\\
\frac{\partial x_j}{\partial r_j}\frac{\partial \lambda_{ij}}{\partial x_j} = -\frac{r_i}{r_j^2}\frac{\partial \lambda_{ij}}{\partial x_j} < 0, & j>i,
\end{cases}
\]

\noindent by Lemma~\ref{lem:MS}.

\item When $i=j$, $\lambda_{i0}= -\frac{m_0}{r_i^3}$ and $\lambda_{ii}= -\frac{ m_i }{2^{3/2}r_i^3}\zeta_\ell$, which are both increasing in $r_i$. For $0<j<i$, by Lemma~\ref{lem:MS}, we have

\[
\frac{\partial \lambda_{ij}}{\partial r_i}
= m_j\left(\frac3{r_i^4}(\phi_1(y_j)+y\phi_1'(y_j)) + \frac{r_j}{r_i^5}\frac{\partial}{\partial y_j}(\phi_1(y_j)+y\phi_1'(y_j))\right) >0.
\]

For $j>i$, we have

\[
\frac{\partial \lambda_{ij}}{\partial r_i}=
\frac{m_j}{r_j^3}\frac{\partial x_j}{\partial r_i} \frac{\partial }{\partial x_j}\left(\frac{\phi_1'(x_j)}{x_j} \right)
=\frac{m_j}{r_j^4}\frac{x_j\phi_1''(x_j)-\phi_1'(x_j)}{x^2_j}.
\]

By Lemma~\ref{lem:MS}, $\phi_1(x)$ is analytic on $(0,1)$ so $\phi_1(x)= \sum_{n\geq0} a_n x^n$ with $a_n\geq 0$. Moreover, $\sum_{k=0}^{\ell-1}e^{i\theta_k}=0$ (since it is the sum of the roots of $z^\ell-1=0$), yielding $a_1 = \phi_1'(0)= \sum_{k=0}^{\ell-1}\cos \theta_k =0$. Hence,

\[
x\phi_1''(x)-\phi_1'(x)=\sum_{n\geq2}a_n n(n-1)x^{n-1}-\sum_{n\geq1}a_n n x^{n-1}=\sum_{n\geq2}a_n n(n-2)x^{n-1}>0.
\]

\item Let $k>j>i$ and define $x_s=r_s/r_k$ for $s=i,j$. The derivative according to $r_k$ is

\[
\frac{\partial \lambda_{sk}}{\partial r_k}= \frac{m_k}{r_s^3}\frac{\partial x_s}{\partial r_k}\frac{\partial}{\partial x_s}\left(x_s^2\phi_1'(x_s)\right) = -\frac{m_k}{r_k^4}\frac{x_s\phi_1''(x_s)+2\phi_1'(x_s)}{x_s}<0.
\]

Now, because $a_1=0$, $(x\phi_1''(x)+2\phi_1'(x))/x$ is strictly positive and increasing. Since $x_i<x_j$, we deduce that $0>\partial_{r_k} \lambda_{ik}>\partial_{r_k} \lambda_{jk}$.

\item Let $0<k<j<i$ and $y_s=r_k/r_s$ for $s=i,j$. We have

\[
\frac{\partial \lambda_{s k}}{\partial r_k} = -\frac{m_k}{r_s^4} (2\phi_1'(y_s)+y_s\phi_1''(y_s))<0.
\]

The function $y\phi_1''(y)+2\phi_1'(y)$ is strictly positive and increasing. Since $y_j<y_i$, we get $\partial_{r_k} \lambda_{ik}<\partial_{r_k} \lambda_{jk}<0$.
\end{enumerate}
\end{proof}

\begin{corollary}\label{cor:dico}
Let $\Lambda = (\Lambda_1,\dots,\Lambda_{n-1}) \in \R^{n-1}$ with

\[
\Lambda_i(r)=\lambda_i(r) - \lambda_{i+1}(r),\qquad i= 1,\dots,n-1.
\]

We have the four following properties:
\begin{enumerate}
    \item $\partial_{r_i}\Lambda_i>0$ for all $i$.
    \item $\partial_{r_{i+1}}\Lambda_i<0$ for all $i$.
    \item If $j>i+1$, then $\partial_{r_{j}}\Lambda_i>0$.
    \item If $j<i$, then $\partial_{r_{j}}\Lambda_i<0$.
\end{enumerate}
\end{corollary}


\begin{lemma}\label{lem:p0}
Let $n\in\{2, 3, 4\}$, $\ell \in \N_{>0}$ and $(m_0,m)\in \R_{\geq0}\times \R^n_{>0}$. There exists $p \in \mathcal{R}^n$, such that the $n\times \ell + 1$ spiderweb configuration respects

\[
\lambda_1(p) < \ldots < \lambda_n(p).
\]

Furthermore, in the case $n = 4$, we may choose $(p_3,p_4)$ such that

\[
\lambda_3(p_1,r_2,p_3,p_4) < \lambda_4(p_1,r_2,p_3,p_4),
\qquad \forall r_2\in (p_2,p_3].
\]
\end{lemma}

\begin{proof}
We start with an initial circle located at $p_1\in \R_{>0}$. For $n=2$, we have $\lambda_1(p_1,+\infty)<\lambda_2(p_1,+\infty)=0$ and from the Point $4$ of Lemma~\ref{lem:dico}, this inequality is preserved when $r_2$ decreases. Repeating this exact argument for $n=3,4$ gives the expected result. Moreover, in the case $n=4$, the radius $p_4$ can be taken sufficiently large so that $\lambda_3(p_1,r_{2},p_3,p_4) < \lambda_4(p_1,r_{2},p_3,p_4)$ for all $r_2\in (p_1,p_2]$.
\end{proof}

\subsection{Equations for spiderweb central configurations}

We have seen in Section \ref{sec:scalings} that the common $\lambda$ characterizing a central configuration may be fixed to any real strictly negative number, independently of $n,\ell $ and the mass distribution $(m_0,m) $.

Hence, depending of the method we will use, a spiderweb central configuration can be seen as a solution of $\Lambda=0$ or as a zero of the map $f : \R^{n} \longrightarrow \R^n$ given by

\begin{align}\label{fiphi}
&f_i(r_i,y_1,\dots,y_{i-1},x_{i+1},\dots,x_n)\nonumber\\
&= \lambda \, r_i  -  \frac{F_i(r_i,y_1,\dots,y_{i-1},x_{i+1},\dots,x_n)}{m_i}\nonumber\\
&=
\lambda \, r_i + \frac{ m_i }{2^{3/2}r_i^2}\zeta_{\ell} +\frac{ m_0 }{r_i^2} + \sum_{j=1}^{i-1} \frac{m_j}{r_i^2}\left(\phi_1(y_j)+y_j\phi_1'(y_j)\right) - \sum_{j=i+1}^n \frac{m_j x_j^2}{r_i^2} \phi_1'(x_j).
\end{align}

\section{Existence of spiderweb central configurations with arbitrary $n$ and $\ell$}\label{sec:general}
In this section we give a very short proof of the theorem announced in \cite{CDL}. This requires introducing the tool of restricted spiderweb central configurations, which will be used also later in the paper.

\subsection{Restricted spiderweb central configurations}

\begin{theorem}\label{thm:restreint}
Let $n\in \N$, $\ell \in \N_{\geq2}$ and $(m_0,m)\in \R_{\geq0}\times \R_{>0}^n$. Suppose there exists some radii $r\in \mathcal{R}^{n}$ giving a $n \times \ell +1$ spiderweb central configuration. Then, for any $i\in \{ 0,\dots,n \}$, if $m_{n+1}=0$, there exists a unique position

\[
r_{n+1} \in
\begin{cases}
(r_{i},r_{i+1}), & \text{if }i = 0,\dots,n-1,\\
(r_{n},+\infty), & \text{if } i =n,
\end{cases}
\]

\noindent giving a $(n+1)\times \ell+1$ spiderweb central configuration.
\end{theorem}

\begin{proof}
By hypothesis, there exists $r\in \mathcal{R}^{n}$ such that the $n\times \ell +1$ spiderweb configuration is central.

Fix $i\in \{ 0,\dots,n \}$ and add a $(n+1)^{th}$ circle of zero mass and of radius

\[
r_{n+1} \in
\begin{cases}
(r_{i},r_{i+1})\,, & \text{if } i<n,\\
(r_{n},+\infty)\,, & \text{if } i=n,
\end{cases}
\]

Consider

\[
\lambda_{j}(r,r_{n+1})= \frac{F_{j}(r,r_{n+1})}{m_{j} r_{j}}\,,\quad j=1,\dots,n+1,
\]

\noindent which, by \eqref{Fij}, is perfectly defined when $m_{n+1}=0$.

On the one hand, adding particles of negligible mass bears no effect on the gravitational force felt by the particles on the $n$ initial circle, that is $\lambda_1=\ldots=\lambda_{n}=\lambda<0$.

On the other hand, Lemma~\ref{lem:dico} gives the monotonous limits

\[
\quad \lim_{r_{n+1} \, \searrow \, r_{i}} \lambda_{n+1}(r,r_{n+1}) = -\infty,\quad
\begin{cases}
\lim_{r_{n+1} \, \nearrow \, r_{i+1}} \lambda_{n+1}(r,r_{n+1}) = +\infty, & \text{if } i<n+1,\\
\lim_{r_{n+1} \rightarrow +\infty} \lambda_{n+1}(r,r_{n+1}) = 0^-, & \text{if } i=n,
\end{cases}
\]

Consequently, there is a unique $r_{n+1}$ in each case such that $\lambda_{n+1}=\lambda$.
\end{proof}

\subsection{Proof of the existence of central configurations}

\begin{theorem}\label{thm:CDL}
Let $n \in \N$, $\ell \in \N_{\geq2}$ and $(m_0,m_1)\in \R_{\geq0}\times\R_{>0}$. There exists $r\in \mathcal{R}^{n}$ and masses $m_n \ll \ldots \ll m_2\ll m_1$ giving  a $n\times \ell+1$  spiderweb central configuration. \end{theorem}

\begin{proof}
The proof is by induction on $n$. Let $m_0\in \R_{\geq0}$  and $\lambda <0$. If $n=1$,  for any $m_1 \in \R_{>0}$, according to equation \eqref{fiphi}, there exists a unique zero $f(r_1)=f_1(r_1)=0$ and the derivative never vanishes according to Lemma~\ref{lem:MS}. 

Let $n\geq 2$, and suppose that the jacobian $|D_{(r_1,\dots,r_{n-1})} f|$ is invertible for $n-1$ circles with $m_1 \gg \ldots \gg m_{n-1}$. We place a $n^{th}$ circle with $m_n=0$ and, by theorem \ref{thm:restreint}, there exists a unique $r(0) = (r_1,\dots,r_{n-1},r_n) \in \mathcal{R}^{n}$ giving a spiderweb central configuration.

Using the notation $\partial_k = \frac{\partial}{\partial r_k}$, we have

\begin{equation}\label{jacobien_DDS}
\begin{split}
\partial_i f_i &=\\
&{\tiny \lambda - \frac{ m_i }{r_i^3\sqrt{2}}\zeta_{\ell} - \frac{ 2m_0 }{r_i^3} - \sum_{j=1}^{i-1} \frac{m_j}{r_i^3}\left(2\phi_1(y_j)+4y_j\phi_1'(y_j)+y_j^2\phi_1''(y_j)\right)-\sum_{j=i+1}^n\frac{m_j x_j^3}{r_i^3}\phi''_1(x_j)},\\
\partial_j f_i &=
\begin{cases}
\displaystyle \frac{m_j}{r_i^3}\left( 2\phi'_1(y_j)+y_j\phi''_1(y_j)\right), & j<i,\\
\displaystyle \frac{m_j x_j^3}{r_i^3}\left( 2\phi'_1(x_j)+x_j\phi''_1(x_j)\right), & j>i.
\end{cases}
\end{split}
\end{equation}

Therefore,

\[
|D_r f(r(0))|=
\left|\left(
\begin{array}{c c c ;{2pt/2pt} c}
&  & & 0\\
& D_{(r_1,\dots,r_{n-1})} f(r(0)) & & \vdots\\
& & & 0\\
& & & \\
\hdashline[2pt/2pt]
& & &\\
\partial_{1}f_n(r(0)) & \dots & \partial_{n-1}f_n(r(0))  & \partial_{n}f_n(r(0))
\end{array}\right)\right|\neq 0,
\]

\noindent because $\partial_{n}f_n(r(0))<0$ by Lemma~\ref{lem:MS} and $|D_{(r_1,\dots,r_{n-1})} f(r(0))|\neq 0$ by hypothesis.

The \emph{implicit function theorem} yields a neighborhood $V$ of $m_n=0$ such that the function $r=r(m_n)$ is a zero of $f$ for all $m_n \in V$. So, the condition $m_n \ll m_{n-1}$ ensures the existence of the spiderweb central configuration.
\end{proof}

\section{Existence and uniqueness for circles of low density ($\ell$ small)}\label{sec:analy2}

Recall the map $f$ given in \eqref{fiphi}, whose zeros give a spiderweb central configuration, and its Jacobian matrix $D_r f$ given in  \eqref{jacobien_DDS}. In Theorem~\ref{thm:CDL}, the existence of a spiderweb central configuration is asserted under the condition $m_1\gg \ldots \gg m_n$. However, for a system with circles of low density, namely small values of $\ell$, the \emph{implicit function theorem} may be used to extend the zeros of $f$ for all positive value of the mass $m_n$ on the outermost circle. In such cases, iterating the argument allows us to construct a spiderweb central configuration for an arbitrary number $n$ of circles and, moreover, to prove its uniqueness.

\begin{theorem}\label{thm:ell}
Let $n \in \N$, $\ell \in \{2,\dots,9\}$ and $(m_0,m) \in\R_{\geq0} \times \R^n_{>0}$. For a fix $\lambda$, there exists a unique $r\in \mathcal{R}^{n}$ such that the $n\times \ell+1$ spiderweb configuration is central.

Furthermore, if $m_1\geq\ldots\geq m_n$, then the $n\times \ell+1$ spiderweb central configuration exists and is unique in the following cases:
\begin{enumerate}
    \item $\ell=10$ and $n\leq 17$.
    \item $\ell=11$ and $n\leq 9$.
    \item $\ell=12$ and $n\leq 6$.
    \item $\ell=13$ and $n\leq 5$.
    \item $\ell=14,15$ and $n\leq 4$.
    \item $\ell=16,17,18$ and $n\leq 3$.
\end{enumerate}
\end{theorem}

\begin{proof}
For $n=1$, we have a regular $\ell$-gon with a central mass. The equation \eqref{fiphi} shows the existence of a unique $r_1 \in \R_{>0}$ such that the configuration is central.

Suppose for $n-1$ circles, with $(m_1,\dots,m_{n-1}) \in \R^{n-1}_{>0}$, there exists a unique $(r_1,\dots,r_{n-1})\in \mathcal{R}^{n-1}$ such that the $(n-1) \times \ell + 1$ mass-particles form a spiderweb central configuration for $\lambda$. By Theorem~\ref{thm:restreint}, there exists a unique $r(0)=(r_1,\dots,r_{n-1},r_n) \in \mathcal{R}^{n}$ such that the $n \times \ell + 1$ spiderweb configuration is central for this $\lambda$.

We will prove the following:

\begin{enumerate}
    \item[$\mathit{1.}$] \textit{The jacobian matrix $D_r f \in M_n(\R)$, whose entries are given by \eqref{jacobien_DDS}, is invertible for all $r \in \mathcal{R}^{n}$ and $m_n \in \R_{\geq0}$.}
    
    \item[$\mathit{2.}$] \textit{All the radii remain bounded for all $m_n\in \R_{\geq0}$.}
    
    \item[$\mathit{3.}$] \textit{All the radii remain distinct for all $m_n\in \R_{\geq0}$.}
 
     \end{enumerate}

The first claim allows using the \emph{implicit function theorem} to obtain a function $r(m_n)$ such that $f(r(m_n),m_n)=0$. Claims 2 and 3  allow concluding  that $r(m_n)$ can be uniquely extended for any value of $m_n \in \R_{\geq0}$, yielding its local uniqueness. The global uniqueness follows from the following argument: suppose there is an other function $\psi(m_n)$ such that $f(\psi(m_n),m_n)=0$, then it can be extended on $\R_{\geq0}$. In particular, $\psi(0)=r(0)$ because $r(0)$ is unique. Hence, $\psi(m_n) = r(m_n)$ for every $m_n\in \R_{\geq0}$.

\bigskip

\begin{itemize}
    \item[\textbf{Claim} $1$:]
    
Recall that a sufficient criterion for a matrix to be invertible is to be \emph{strictly diagonally dominant}\footnote{i.e. $|M_{ii}|> |\sum_{j\neq i} M_{ij} |$ for $i=1,\dots,n$ and $M \in M_n(\R)$.}. We know that $\zeta_{\ell}$ is strictly positive and, by lemma \ref{lem:MS}, that $\partial_i f_i<0$ and $\partial_j f_i>0$ for $j\neq i$. Hence, we must show

\[
-\partial_i f_i - \sum_{\begin{smallmatrix}j=1\\j\neq i \end{smallmatrix}}^n \partial_j f_i>0, \quad i=1,\dots,n.
\]

Now, if we rewrite $f_i$ in terms of $x_j$ and $\phi$ defined respectively in \eqref{def:x} and \eqref{def:phi}

\begin{equation*}
- \partial_i f_i - \sum_{\begin{smallmatrix}j=1\\j\neq i \end{smallmatrix}}^n \partial_j f_i
=-\lambda +\frac{ m_i }{r_i^3\sqrt{2}}\zeta_{\ell} +\frac{2m_0 }{r_i^3}
+ \sum_{\begin{smallmatrix}j=1\\j\neq i \end{smallmatrix}}^n \frac{m_j x_j^3}{r_i^3}\left(\phi''_1(x_j)(1-x_j)-2\phi'_1(x_j)\right) .
\end{equation*}

But,

\begin{align*}
\phi_1'(x_j)
&= \frac{d}{d x_j}\left( \sum_{k=0}^{\ell-1} \frac{1}{\left(1+x_j^2-2x_j\cos \theta_k\right)^{1/2}}\right)
= - \sum_{k=0}^{\ell-1} \frac{x_j-\cos \theta_k}{\left(1+x_j^2-2x_j\cos \theta_k\right)^{3/2}}\,,\\
\phi''_1(x_j)
&=\frac{d}{d x_j}\left(- \sum_{k=0}^{\ell-1} \frac{x_j-\cos \theta_k}{\left(1+x_j^2-2x_j\cos \theta_k\right)^{3/2}}\right)
=-\sum_{k=0}^{\ell-1} \frac{1+x_j^2-2x_j \cos \theta_k -3(x_j-\cos \theta_k)^2}{\left(1+x_j^2-2x_j\cos \theta_k\right)^{5/2}},
\end{align*}

\noindent whence


\[
\phi''_1(x_j)(1-x_j)-2\phi'_1(x_j)=\sum_{k=0}^{\ell-1} \frac{(1-\cos \theta_k)(2x_j^2+x_j(3-\cos \theta_k)-1-3\cos \theta_k)}{\left(1+x_j^2-2x_j\cos \theta_k\right)^{5/2}}.
\]

Noticing that the term for $k=0$ is zero, we find

\begin{equation}\label{eqDDS}
- \partial_i f_i - \sum_{\begin{smallmatrix}j=1\\j\neq i \end{smallmatrix}}^n \partial_j f_i
=-\lambda +\frac{ m_i }{r_i^3\sqrt{2}}\zeta_{\ell} +\frac{ 2m_0 }{r_i^3}
+ \sum_{\begin{smallmatrix}j=1\\j\neq i \end{smallmatrix}}^n \frac{m_j x_j^3}{r_i^3} h_\ell(x_j) ,
\end{equation}

\noindent with

\[
h_\ell(x_j)=
\sum_{k=1}^{\ell-1} \frac{(1-\cos \theta_k)(2x_j^2+x_j(3-\cos \theta_k)-1-3\cos \theta_k)}{\left(1+x_j^2-2x_j\cos \theta_k\right)^{5/2}}.
\]

Since $\lambda <0$ and $\zeta_{\ell}>0$, the expression given in \eqref{eqDDS} is strictly positive if $h_\ell(x_j)$ is positive, where the sign of the latter depends on the sign of the polynomial $2x^2+x(3-\cos \theta_k)-1-3\cos \theta_k$.

It is sufficient to show that the sign of $h_\ell(x_j)$ is strictly positive in the case $x_j\in (0,1)$, that is when $j>i$. Indeed, for $x>1$, we have

\[
2x^2+x(3-\cos \theta_k)-1-3\cos \theta_k> 2(x^2+x)-4 >0.
\]

\begin{figure}[ht]
\centering
\subfigure[][]{
\includegraphics[height=2.3in]{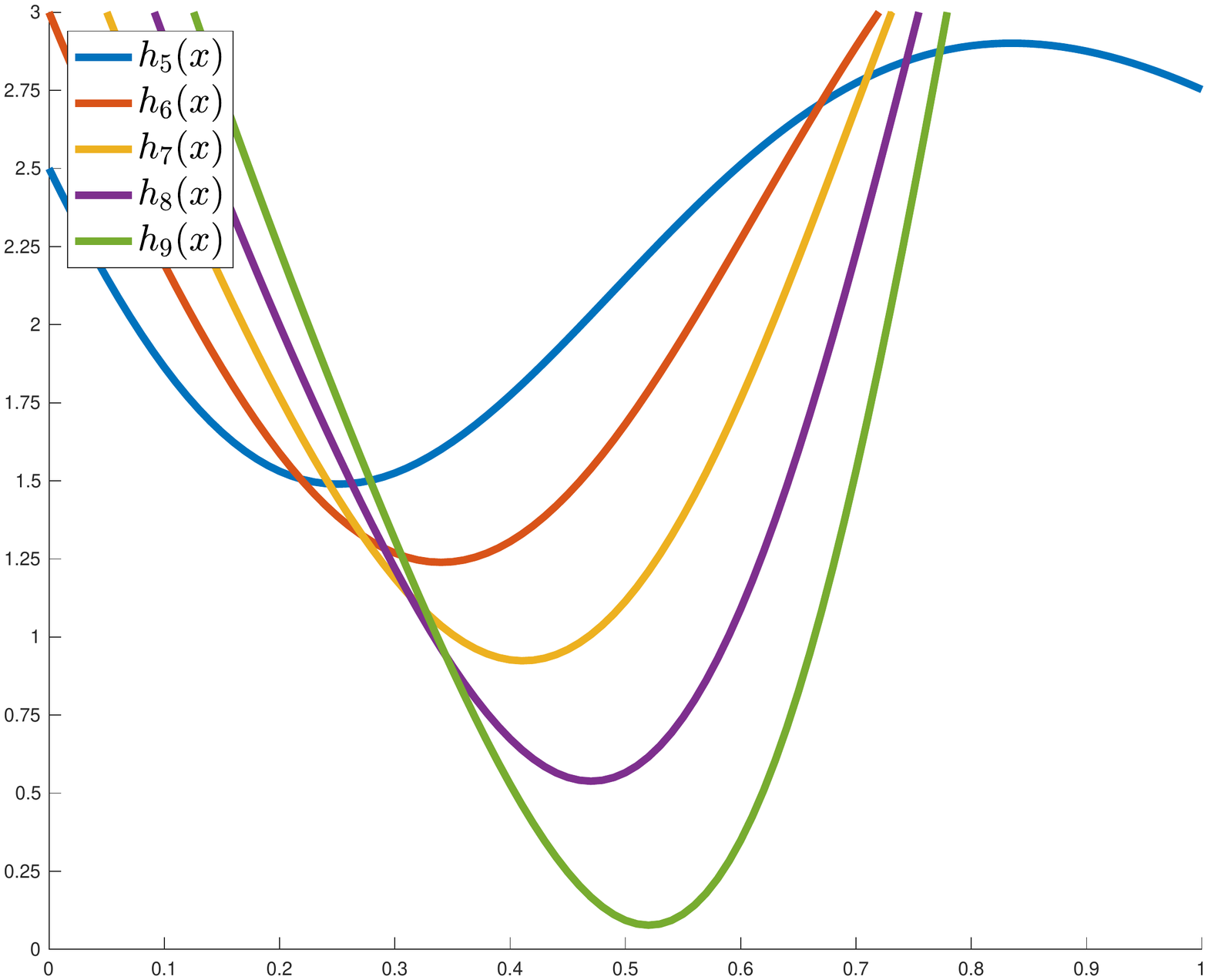}\label{fig:h06_08}}
\hspace{8pt}
\subfigure[][]{
\includegraphics[height=2.3in]{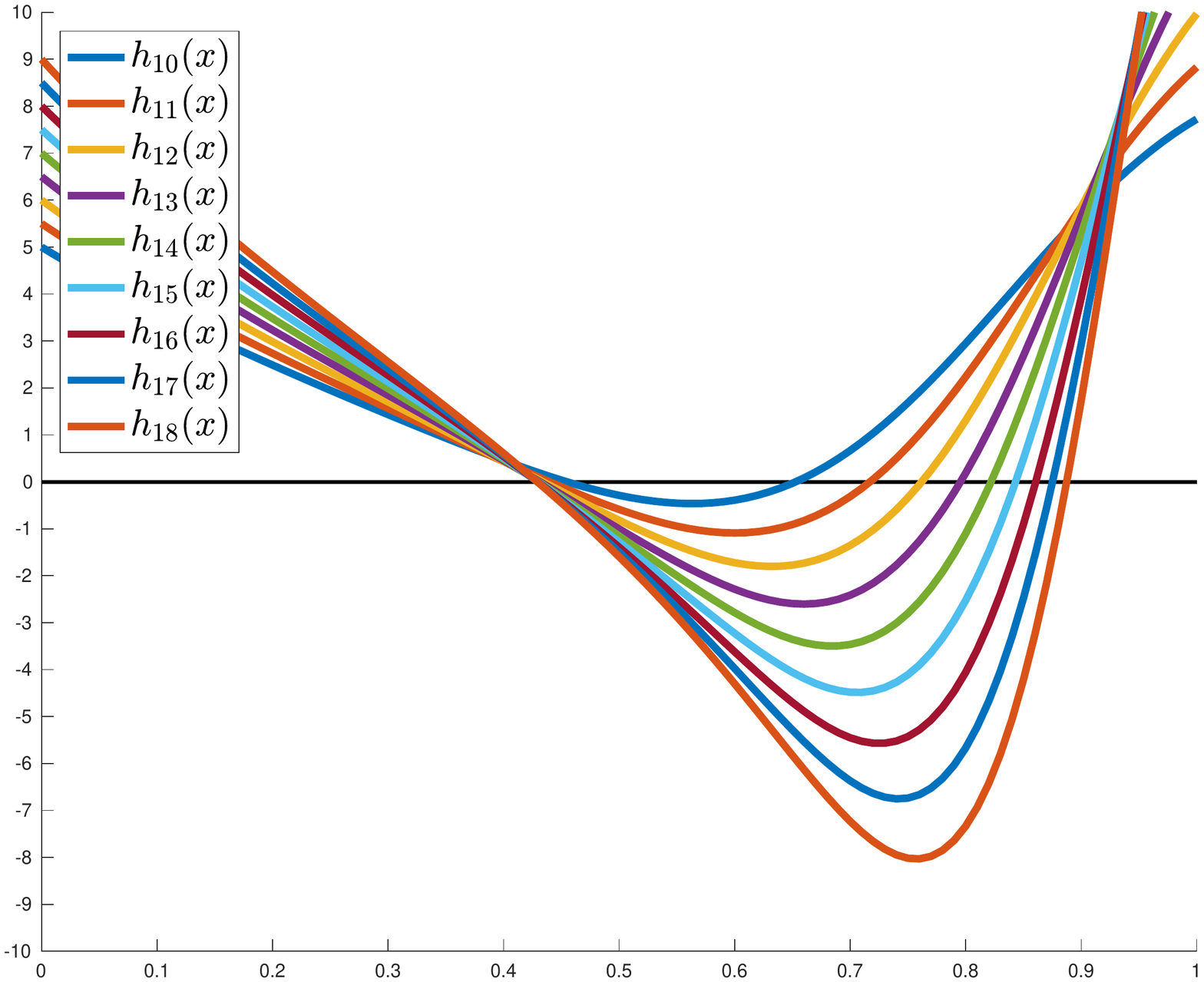}
\label{fig:h10_18}}
\caption{Graph of $h_\ell(x)$ in the interval $[0,1]$ for
\subref{fig:h06_08} $\ell=5,\dots,9$ and
\subref{fig:h10_18} $\ell=10,\dots,18$.}
\label{fig:h}
\end{figure}

\subsubsection*{Cases $\ell=2,3,4$ }

For $\ell=2,3$, we have

\[
h_\ell(x_j) =
\begin{cases}
\displaystyle \frac{4}{(1+x_j)^3}, & \ell=2,\\
\displaystyle \frac{3(\frac{1}{2}+\frac{7}{2}x+2 x^2)}{(1+x_j+x_j^2)^{5/2}}, & \ell=3,
\end{cases}
\]

\noindent which are clearly strictly positive function for all $x_j\in (0,1)$.
    
For $\ell=4$, we have

\[
h_4(x_j) = \frac{4}{(1+x_j)^3}+2\frac{2x_j^2+3x_j-1}{(1+x_j^2)^{5/2}}.
\]

It is sufficient to show that $v(x) = (1+x)^3h_4(x)$ is strictly positive for $x\in (0,1)$. The derivative is $v'(x)=6 x (1 + x)^2 (7 - 3 x^2)(1 + x^2)^{-7/2}$ whose only positive root is strictly greater than $1$. Since $v(0)>0$, it follows that $v(x)>0$ for all $x\in (0,1)$.

\subsubsection*{Cases $\ell=5,\dots,9$}

To avoid fastidious work, we give a simple computer-assisted proof for these cases.

Since $h_\ell \in C^\infty([0,1])$, there is some $M>0$ such that

\[
|h_\ell'(x)| <M, \qquad \forall x \in [0,1].
\]

Let $\bigcup_{q=0}^p [s_q,s_{q+1}]=[0,1]$ with $s_q=q/p$. For any $s \in [s_q,s_{q+1}]$, we have $|h_\ell(s)-h_\ell(s_{q+1})|<M/p\Rightarrow h_\ell(s) >-M/p+h_\ell(s_{q+1})$.

Graphically, we estimate a lower bound $m>0$ for $h_\ell(x)$ on the interval $[0,1]$. We then choose $p$ so that $M/p<m$. Using the interval arithmetic \cite{Ru99a}, we can rigorously check that $h_\ell(s_{q+1})$, for $q=0,\dots,p-1$, is strictly greater than $m$. Consequently, we get $ h_\ell(s) >-M/p+h_\ell(s_{q+1})>0$ for all $s \in [s_q,s_{q+1}]$ and $q=0,\dots,p-1$ (cf. Figure~\ref{fig:h06_08}).

\subsubsection*{Cases $\ell=10,\dots,18$ }

Figure~\ref{fig:h10_18} shows that $h_\ell(x_j)$ is negative for some $x_j\in (0,1)$.

Considering $m_1\geq\ldots\geq m_n$, we have that the equation \eqref{eqDDS} implies
    
\[
\frac{ m_i }{r_i^3\sqrt{2}}\zeta_{\ell} + \frac{ 2m_0 }{r_i^3}
+ \sum_{j=i+1}^n \frac{m_j x_j^3}{r_i^3} h_\ell(x_j)
> \frac{ 2m_0 }{r_i^3}+\frac{ m_{i} }{r_i^3\sqrt{2}}\left(\zeta_{\ell}+\sqrt{2}(n-1)\min_{x\in [0,1]} h_\ell(x)\right).
\]

For each value of $\ell$, we have

\[
\zeta_{\ell}
\geq
\begin{cases}
10.9, & \ell=10,\\
12.45, & \ell=11,\\
14, & \ell = 12,\\
15.74, & \ell = 13,\\
17, & \ell = 14,\\
19.13, & \ell = 15,\\
20.8, & \ell = 16,\\
22, & \ell = 17,\\
24, & \ell = 18,
\end{cases}
\qquad \text{and} \qquad
h_{\ell}(x)
>
\begin{cases}
-0.48, & \ell=10,\\
-1.1, & \ell=11,\\
-1.82, & \ell = 12,\\
-2.61, & \ell = 13,\\
-4, & \ell = 14,\\
-4.5, & \ell = 15,\\
-5.6, & \ell = 16,\\
-7, & \ell = 17,\\
-8.2, & \ell = 18,
\end{cases}\quad \text{for all } x\in [0,1].
\]

It is now easy to check that $n$ cannot be greater than $17,9,6,5,4,4,3,3,3$ for $\ell = 10,\dots,18$ respectively, in order that $\zeta_\ell + \sqrt{2}(n-1) \min_{x\in [0,1]}h_\ell(x)>0$.
\end{itemize}
    
Therefore, we have established that $-\partial_i f_i - \sum_{\begin{smallmatrix}j=1\\j\neq i \end{smallmatrix}}^n \partial_j f_i$ is a sum of positive terms, thus proving the first claim.

\bigskip

By the \emph{implicit function theorem}, we conclude to the existence of a function $r=r(m_n)$ defined on a neighborhood $V$ of $m_n=0$, such that $f(r(m_n),m_n)=0$ for all $m_n\in V$.

\bigskip

\begin{itemize}
\item[\textbf{Claim} $2$:]
Let $\{a_k\}_{k\geq1}$ be a sequence in $V$ such that $\lim_{k\rightarrow +\infty} a_k = \sup V$.

Suppose that $\lim_{k\rightarrow +\infty} r_i(a_k) = +\infty$ for some index $i$. Then, it is necessary that $\lim_{k\rightarrow +\infty} r_{j}(a_k)=+\infty$ for $j< i$ because $f_i(r(a_k),a_k)=0$. Iterating, we need 
$\lim_{k\rightarrow +\infty} r_{j'}(a_k)=+\infty$, for $j'<j$, since $f_{j}(r(a_k),a_k)=0$. Thus, we conclude $\lim_{k\rightarrow +\infty} r_1(a_k)=+\infty$ yielding $\lim_{k\rightarrow +\infty} f_1 (r(a_k),a_k)=-\infty$. Contradiction.

\bigskip

\item[\textbf{Claim} $3$:]
Recall $V$ and $\{a_k\}_{k\geq1}$ previously defined.

Without loss of generality, suppose that for an index $i$ we have $\lim_{k\rightarrow +\infty} r_i(a_k)/r_j(a_k) =1$ with $j> i$. To preserve the equality $f_i(r(a_k),a_k)=0$, the equation \eqref{fiphi} implies $\lim_{k\rightarrow +\infty} r_{j'}(a_k)/r_{i}(a_k) = 1$ with $j'<i$. Again, $f_{j'}(r(a_k),a_k)=0$ requires that $\lim_{k\rightarrow +\infty} r_{j''}(a_k) /r_{j'}(a_k) = 1$ with $j''<j'$. Iterating this argument, we find $\lim_{k\rightarrow +\infty} r_{1}(a_k) / r_2(a_k) = 1$ yielding $\lim_{k\rightarrow +\infty} f_1(r(a_k),a_k) = -\infty$. Contradiction.
\end{itemize}

\end{proof}

\section{Constructive proof of existence for $n\in \{3,4\}$ and arbitrary $\ell$}\label{sec:analy1}

The proof is a completion of Saari's proof given in \cite{Sa2} for the existence of spiderweb central configurations when $n\in\{3,4\}$.  It makes an essential use of all properties proved in Lemma~\ref{lem:dico} and Corollary~\ref{cor:dico}.

\begin{theorem}\label{thm:n}
Let $n\in \N$, $\ell \in \N_{\geq2}$ and $(m_0,m) \in \R_{\geq0}\times \R_{>0}^n$. If $n\leq 4$, then there exists a $n\times \ell +1$ spiderweb central configuration.

Furthermore, in the case $n=2$, it is the unique such configuration for a fix $\lambda$. (The result for $n=2$ is already in \cite{MS}).
\end{theorem}

\begin{proof} 
The proof consists in three parts, \textbf{Part A} (resp. \textbf{Part B} and \textbf{Part C}) sufficient for $n=2$ (resp. for $3$ and $4$).

\subsubsection*{Part A}

Let $\mathcal{N}\in \{2,3,4\}$ and fix $i \in \{1,2,3\}$. By Lemma~\ref{lem:p0}, there exists $p \in \mathcal{R}^{\mathcal{N}}$ such that $\Lambda_i(p)<0$. Moreover, Corollary~\ref{cor:dico} implies that $\lim_{r_{i+1} \searrow \, p_i}\Lambda_i= +\infty$ and is monotonously decreasing. Thus, the function $\Lambda_i$ has a unique zero. Taking $\mathcal{N}=2$ gives the unique spiderweb central configuration.

\subsubsection*{Part B}

Let $n=3$. Lemma~\ref{lem:p0} gives us an initial position such that the three circles satisfies $\lambda_1(p)< \lambda_2(p)<\lambda_3(p)< 0$. We drop the dependency on $p_1$ as we keep the position of the first circle steady.

Now, from \textbf{Part A} we know that we can bring the third circle closer to the second one to a unique $r^*_{3}$ such that

\[
\lambda_1(p_{2},r^*_{3})<\lambda_{2}(p_{2},r^*_{3})=\lambda_3(p_{2},r^*_{3}).
\]

Corollary~\ref{cor:dico} gives $\partial_{r_{3}} \Lambda_{2} < 0$ for all $(r_{2},r_{3})$, so, by the \emph{implicit function theorem}, there is a function $r_{3}=r_{3}(r_{2})$, defined on a neighborhood $V$ of $r_{2}=p_{2}$, such that

\[
r^*_{3}=r_{3}(p_{2}) \qquad \text{and} \qquad \Lambda_{2}(r_{2},r_{3}(r_{2}))=0 , \quad \forall r_{2}\in V.
\]

Since for any smaller $r_2$
\textbf{Part A} holds, the function $\Tilde{r}_{3}$ is unique and can be extended for all $r_{2}\in (p_1,p_{2}]$. 

Moreover, the function $r_{3}(r_{2})$ is strictly increasing since $r_{3}'(r_{2})=-\frac{\partial_{r_{2}}\Lambda_{2}}{\partial_{r_{3}}\Lambda_{2}}>0$. This implies that the limit $\lim_{r_{2}\searrow p_1}\lambda_1(r_{2},r_{3}(r_{2})) = +\infty$ is monotonous. Meanwhile, we can always assert that $\lambda_3<0$.

Thus, from equation \eqref{eq:n3} and the \emph{intermediate value theorem}, there exists a $r_{2}^{sol}\in (p_1,p_{2})$ such that

\[
\lambda_{1}(r_{2}^{sol},r_{3}(r_{2}^{sol}))=\lambda_{2}(r_{2}^{sol},r_{3}(r_{2}^{sol}))=\lambda_{3}(r_{2}^{sol},r_{3}(r_{2}^{sol})).
\]

\subsubsection*{Part C}

Let $n=4$. Once again, by Lemma~\ref{lem:p0} there is an initial position satisfying $\lambda_1(p)< \lambda_2(p)<\lambda_3(p)<\lambda_4(p)< 0$. Also, we keep $p_1$ fixed so that the value of $\lambda_i$ do not depend on $p_1$ anymore.

By \textbf{Part A}, we may bring the fourth circle closer to the third to a unique $r^*_{4}$ such that

\[
\lambda_1(p_{2},p_{3},r^*_{4}) < \lambda_2(p_{2},p_{3},r^*_{4}) < \lambda_3(p_{2},p_{3},r^*_{4})= \lambda_4(p_{2},p_{3},r^*_{4}).
\]

Since we have the inequality $\partial_{r_{4}} \Lambda_{3} < 0$ for all $(r_{2},r_{3},r_{4})$, the \textit{implicit function theorem} gives the existence of a function $r_{4}=z(r_{2},r_{3})$, defined on a neighborhood $V$ of $(r_{2},r_{3})=(p_{2},p_{3})$, such that

\[
z(p_{2},p_{3})=r^*_{4} \quad \text{et} \quad \Lambda_{3}(r_{2},r_{3},z(r_{2},r_{3}))=0 , \quad \forall (r_{2},r_{3}) \in V.
\]

However, we know that $\Lambda_{3}(r_{2},r_{3},p_{4}) <0$ for all $(r_2,r_{3}) \in (p_1,p_{2}]\times(r_{2},p_{3}]$. So, we can use \textbf{Part A} for each pair of values $(r_2,r_3)$. Thus, the function $z(r_{2},r_{3})$ is unique and can be uniquely extended for all $(r_{2},r_{3}) \in (p_1,p_{2}]\times (r_{2},p_{3}]$.

Moreover, the fact that $\frac{d}{d r_2}z(r_{2},p_{3})$ is strictly negative implies that $\Lambda_1$ is monotonously increasing in $r_2$ to $+\infty$.

Additionally, as we follow $z(r_2,r_3)$, we have

\[
\frac{d}{dr_2}\Lambda_2 = \frac{d}{dr_2}\left(\lambda_2 -\lambda_3\right) = \frac{d}{dr_2}\left(\lambda_2 -\lambda_4\right) >0.\]

Hence, $(\Lambda_1(p_{2},p_{3},z(p_{2},p_{3})),\Lambda_2(p_{2},p_{3},z(p_{2},p_{3}))\prec0$, so there exists a unique $\Hat{r}_{2} \in (p_1,p_{2})$, such that we have the inequality

\[
\lambda_1(\nu,z(\nu))=\lambda_{2}(\nu,z(\nu))<\lambda_{3}(\nu,z(\nu))=\lambda_{4}(\nu,z(\nu)),
\]

\noindent where $\nu=(\Hat{r}_{2},p_{3})$.

The point $(\Hat{r}_{2},p_{3})$ is then a zero of the function $\mathcal{L}(r_{2},r_{3},r_{4})=(\Lambda_{1},\Lambda_{3})$ whose jacobian with respect to $(r_{3},r_{4})$ is

\[
\left|D_{(r_{3},r_{4})} \mathcal{L}\right|
=\left|
\begin{pmatrix}
\partial_{r_{3}}\Lambda_{1} & \partial_{r_{4}}\Lambda_{1}\\
\partial_{r_{3}}\Lambda_{3} & \partial_{r_{4}}\Lambda_{3}
\end{pmatrix}
\right|<0.
\]

By the \textit{implicit function theorem}, we have the existence of a function $(r_{3},r_{4})=(r_{3}(r_2),r_{4}(r_2))=\psi(r_{2})$, defined on a neighborhood $V'$ of $r_{2}=\Hat{r}_{2}$, such that

\[
\psi(\Hat{r}_{2}) = (p_{3},z(\Hat{r}_{2},p_{3})) \quad \text{et} \quad \mathcal{L}(r_{2},\psi(r_{2}))=0,\quad \forall r_{2}\in V'.
\]

A quick look at the signs in

\[
\psi'(r_{2})=
\begin{pmatrix}
r_{3}'(r_{2})\\
r_{4}'(r_{2})
\end{pmatrix}
=
\frac{-1}{|D_{(r_{3},r_{4})}\mathcal{L}|}
\begin{pmatrix}
\partial_{r_{4}}\Lambda_{3} & -\partial_{r_{4}}\Lambda_{1}\\
-\partial_{r_{3}}\Lambda_{3} & \partial_{r_{3}}\Lambda_{1}
\end{pmatrix}
\begin{pmatrix}
\partial_{r_{2}}\Lambda_{1}\\
\partial_{r_{2}}\Lambda_{3}
\end{pmatrix}
\]

\noindent shows that $r'_{3}(r_{2})>0$. Whence, for any $r_2\in V' \cap (p_1,\hat{r}_2]$, we get $r_{3}(r_{2}) \in (r_{2},p_{3}]$. Thus, $(r_2,\psi(r_2))$ is contained within the surface given by $r_4 = z(r_{2},r_{3})$ for $(r_{2},r_{3}) \in (p_1,p_{2}]\times (r_{2},p_{3}]$, and this surface is included in $\mathcal{R}^3$ by construction.
Hence, provided $\inf V'> p_1$, the radius of the third and fourth circles, given by $\psi(r_2)$, cannot tend to the same limit as $r_2 \rightarrow \inf V'$.

Since the sign of the Jacobian is always strictly negative, the function $\psi(r_{2})$ may be uniquely extended for all $r_{2} \in (p_1,p_{2}]$. Notice that $\psi$ is unique since $\Hat{r}_2$ is unique.

On the one hand, we have seen that $\lambda_1(\Hat{r}_{2},\psi(\Hat{r}_{2}))=\lambda_{2}(\Hat{r}_{2},\psi(\Hat{r}_{2}))<\lambda_{3}(\Hat{r}_{2},\psi(\Hat{r}_{2}))=\lambda_{4}(\Hat{r}_{2},\psi(\Hat{r}_{2}))$. On the other hand, we always have the limit $\lim_{r_{2}\searrow p_1}\lambda_1(r_{2},\psi(r_{2})) = +\infty$ while $\lambda_4$ remains strictly negative. Consequently, the \emph{intermediate value theorem} implies the existence of $r^{sol}_{2}\in (p_i,p_{2})$ such that

\[
\lambda_1(r_{2}^{sol},\psi(r_{2}^{sol}))=\lambda_{2}(r_{2}^{sol},\psi(r_{2}^{sol}))=\lambda_{3}(r_{2}^{sol},\psi(r_{2}^{sol}))=\lambda_{4}(r_{2}^{sol},\psi(r_{2}^{sol})).
\]
\end{proof}

\begin{remark} The proof requires multiple uses of the implicit function theorem and, in each case, it was easy to show that the corresponding Jacobian had a fixed sign for all values of the $r_i$. 
Going to $n>4$, there seems no easier way than lengthy calculations for each particular value of $n$ to check the hypotheses of the implicit function theorem each time it is necessary to extend a solution by varying the $r_i$.\end{remark}

\section{Computer-assisted proof}\label{sec:numer}

Once again, we consider the map $f$ defined in the Section~\ref{fiphi}, that is

\begin{equation*}
f_i(r) =
\lambda \, r_i + \frac{ m_i }{2^{3/2}r_i^2}\zeta_{\ell} + \frac{ m_0}{r_i^2} + \sum_{\begin{smallmatrix}j=1\\j\neq i \end{smallmatrix}}^{n} \sum_{k=0}^{\ell-1} \frac{m_j( r_i - r_j \cos \theta_k )}{( r_i^2 + r_j^2 - 2 r_i r_j \cos \theta_k )^{3/2}} , \qquad i =1,\, \dots,\, n.
\end{equation*}

Let $A:\R^n\rightarrow \R^n$ be a linear operator and define $T:U\longrightarrow X$ by $T(x) := x - A f(x)$ where $U$ is an open set of $\R^n$.

Knowing an approximate zero of $f$, the \emph{radii polynomial approach} gives bounds so that we may find a ball, centered at this approximation, on which $T$ is a contraction to which we can apply the \emph{Banach fixed point theorem} and $A$ is non singular. Hence, it allows proving the existence and uniqueness of a true solution $\tilde{r}\in \R^n$ lying in this ball.

Due to the singularities in the equation \eqref{eq_newton} we must be careful in our numerical approach. We consider the \emph{local version in finite dimension of the radii polynomial approach} established by Lessard, that is we introduce an upper bound $\rho_*$ for the radius of the ball in order to remain away from any singularities.

\begin{theorem}[Radii polynomial \cite{Lessard}]\label{thm:radii}
Let $U$ an open set of $\mathbb{R}^n$ and $f:U \rightarrow \mathbb{R}^n$ a map of class $C^1$. Let $\bar{x} \in U$ and $A \in M_n(\mathbb{R})$. Let $\rho_*>0$ such that $\overline{B_{\rho_*}(\bar{x})} \subset U$. Let $Y_0,Z_0\in \R$, and $Z_2 : (0,\rho_*]\rightarrow[0,\infty)$ satisfying

\begin{align*}
\|A f(\bar{x})\|&\leq Y_0,\\
\|Id-A D_x f(\bar{x})\|&\leq Z_0,\\
\|A [D_x f(c)-D_x f(\bar{x})]\|&\leq Z_2(\rho)\rho, \quad \forall c \in \overline{B_{\rho}(\bar{x})} \text{ et } \rho \in (0,\rho_*].
\end{align*}
    
Define the radii polynomial by
    
\begin{equation}
p(\rho) := Z_2(\rho)\rho^2-(1-Z_0)\rho+Y_0.
\end{equation}
    
If there exists $\rho_0 \in (0,\rho_*]$ such that $p(\rho_0)<0$, then $A$ is invertible and there exists a unique $\tilde{x} \in \overline{B_{\rho_0}(\bar{x})}$ satisfying $f(\tilde{x})=0$.
\end{theorem}

\begin{remark}
The choice of $\rho_*$ is quite arbitrary. We make an intial heuristic choice and check \textit{a posteriori} that if there exists $\rho_0$ such that $p(\rho_0)<0$ then $\rho_0 < \rho_*$. Otherwise, we must increase the value of $\rho_*$.
\end{remark}

The quantity $\bar{x}$ corresponds to a numerical approximation of the zero of $f$ obtained via Newton's method.

\subsection{Computation of the bounds}

Accordingly to Newton's method, we choose $A$ to be the numerical inverse of $D_{r}f(\bar{r})$ where $\bar{r}$ is the numerical zero of $f$.

To rigorously compute the bounds, we use techniques of interval arithmetic  \cite{Ru99a}. Let us work in $(\R^n,\|\cdot\|_\infty)$. The bounds $Y_0$ and $Z_0$ can be found immediately from the theorem, knowing that the Jacobian matrix $D_{r} f$ is given by

\[
\partial_j f_i =
\begin{cases}
\displaystyle \lambda - \frac{ m_i }{r_i^3\sqrt{2}}\zeta_{\ell} - \frac{ 2m_0 }{r_i^3}
-\sum_{\begin{smallmatrix}j=1\\j\neq i \end{smallmatrix}}^{n} \frac{m_j}{2}\sum_{k=0}^{\ell-1}
\frac{4r_i^2+r_j^2-8 r_i r_j \cos \theta_k + 3 r_j^2 \cos 2\theta_k}{( r_i^2 + r_j^2 - 2 r_i r_j \cos \theta_k )^{5/2}},& j=i,\\
\displaystyle -\frac{m_j}{2} \sum_{k=0}^{\ell-1}
\frac{-4(r_i^2+r_j^2)\cos \theta_k + r_i r_j (7+ \cos 2\theta_k)}{( r_i^2 + r_j^2 - 2 r_i r_j \cos \theta_k)^{5/2}}, & j\neq i.
\end{cases}
\]

In the infinity norm, it is possible to obtain a general expression for the $Z_2$ bound, assuming $f\in C^2$, by applying the mean value theorem.

\begin{lemma}\label{lem:z2}
\[
Z_2=\sup_{b \in \overline{B_{\rho_*}(\bar{x})}}\left(
\max_{1\leq i \leq n} \sum_{1\leq k,m\leq n} \left|
\sum_{1\leq j \leq n} A_{ij} \partial^2_{km}f_j(b)
\right|
\right).
\]
\end{lemma}

We use this lemma to estimate the bound. The tensor $D^2_{r}f$ is given by

\[
\partial^2_{lj}f_i
=\tiny{
\begin{cases}
\displaystyle \frac{3m_i}{\sqrt{2} r_i^4} \zeta_{\ell} + \frac{ 6m_0 }{r_i^4}
+ \sum_{\begin{smallmatrix}j=1\\j\neq i \end{smallmatrix}}^{n} \frac{3m_j}{2}\sum_{k=0}^{\ell-1}
\frac{(r_i - r_j \cos \theta_k) (4 r_i^2 - r_j^2 - 8 r_i r_j \cos \theta_k + 5 r_j^2 \cos 2\theta_k)}{(r_i^2 + r_j^2 - 2 r_i r_j \cos \theta_k)^{7/2}}, & l=j=i,\\
\displaystyle -\frac{3m_j}{4}\sum_{k=0}^{\ell-1}\frac{(r_i (8 r_i^2 + 23 r_j^2) \cos \theta_k - r_j (20 r_i^2 + 2 r_j^2 + (4 r_i^2 + 6 r_j^2) \cos 2\theta_k - r_i r_j \cos 3 \theta_k))}{(r_i^2 + r_j^2 - 2 r_i r_j \cos \theta_k)^{7/2}}, & l\neq j=i \text{ ou } j\neq l=i,\\
\displaystyle -\frac{3m_j}{4}\sum_{k=0}^{\ell-1}\frac{(r_j ( 8 r_j^2+23 r_i^2) \cos \theta_k - r_i ( 20 r_j^2 + 2 r_i^2 + (4 r_j^2 + 6 r_i^2) \cos 2\theta_k - r_i r_j \cos 3 \theta_k))}{(r_i^2 + r_j^2 - 2 r_i r_j \cos \theta_k)^{7/2}}, & l=j\neq i,\\
\displaystyle 0, & l\neq j \neq i.
\end{cases}}
\]

We unfold $D^2_r f$ and represent it by the $n\times n^2$ matrix $B(b)=(B_1(b)\mid\dots\mid B_n(b))$, where

\[
B_i =
\begin{pmatrix}
-\partial_{1i}^2\frac{F_1(b)}{m_1} & \dots &  -\partial_{ji}^2\frac{F_1(b)}{m_1} & \dots &
-\partial_{ni}^2\frac{F_1(b)}{m_1}\\
\vdots & & \vdots & & \vdots \\
-\partial_{1i}^2\frac{F_k(b)}{m_k} & \dots &  -\partial_{ji}^2\frac{F_k(b)}{m_k} & \dots &
-\partial_{ni}^2\frac{F_k(b)}{m_k}\\
\vdots & & \vdots & & \vdots \\
-\partial_{1i}^2\frac{F_n(b)}{m_n} & \dots &  -\partial_{ji}^2\frac{F_n(b)}{m_n} & \dots &
-\partial_{ni}^2\frac{F_n(b)}{m_n}\\
\end{pmatrix}, \qquad  i =1,\dots,n.
\]

Whence,

\[
Z_2
= \sup_{b\in\overline{B_{\rho_*}(\bar{r})}}\|A B(b)\|_{\infty}
= \sup_{b \in \overline{B_{\rho_*}(\bar{r})}}\left(
\max_{1\leq i \leq n} \sum_{1\leq k,m\leq n} \left|
\sum_{1\leq j \leq n} A_{ij} \partial^2_{km}f_j(b)
\right|
\right).
\]

\subsection{Numerical experimentations with circles of equal mass}

We have tested successfully our algorithm with $\lambda=-1$ for all $n\leq 100$, $\ell \leq 200$ even when $m_0=0$ and $m_1=\ldots=m_n$. The bounds have been rigorously computed using interval arithmetic. 

From our many numerical experimentations we are led to believe that the  $n \times \ell$ and $n \times \ell+1$ spiderweb central configurations not only exist, but are unique in the sense of Section \ref{sec:scalings}. Saari stated the same result in his papers \cite{Sa1} and \cite{Sa2}.

\section{Mass distribution}\label{sec:mass}

The numerical approach allows quantitative insights on spiderweb central configurations. All the profiles studied in this section are validated by applying Theorem~\ref{thm:radii} to each spiderweb central configuration.

For this purpose we introduce some invariants of the configurations. The first is the \emph{relative spacing} between consecutive circles (see Figure~\ref{fig:radii})

\[
a_i = a_i(\ell)= \frac{r_{i+1}(\ell)-r_i(\ell)}{r_1(\ell)}.
\]

The second is the  \emph{relative width} of a spiderweb central configuration  given by (see Figure~\ref{fig:width})

\[
b = b(\ell) = \sum_{i=1}^{n-1}a_i(\ell)= \frac{r_{n}(\ell)-r_1(\ell)}{r_1(\ell)}.
\]

Depending on the context, we write explicitly the dependence on $\ell$.

\begin{figure}
\centering
\subfigure[$\ell=2$]{
\label{fig:radii-a}
\includegraphics[height=2in]{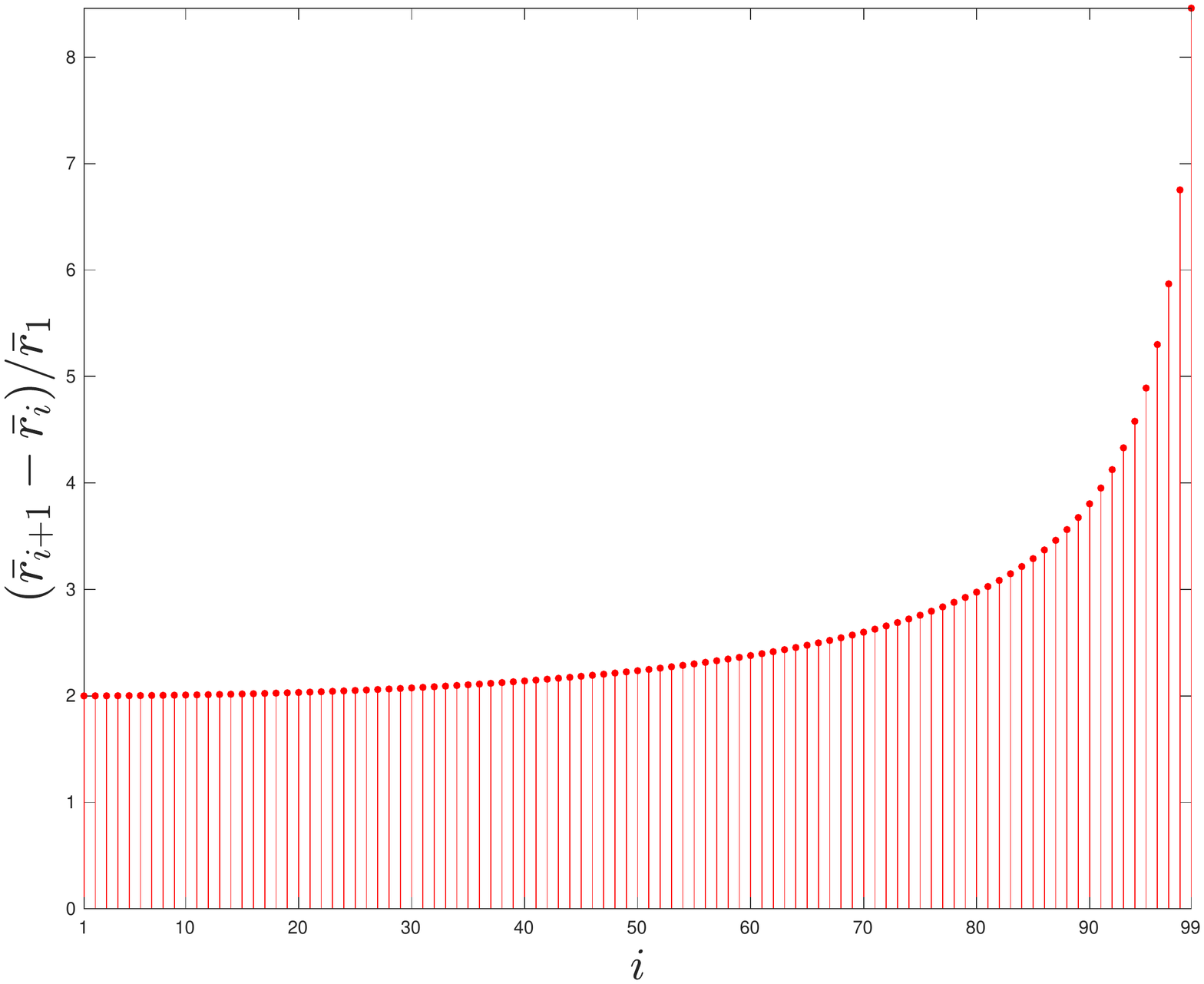}}
\hspace{8pt}
\subfigure[$\ell=6$]{
\label{fig:radii-b}
\includegraphics[height=2in]{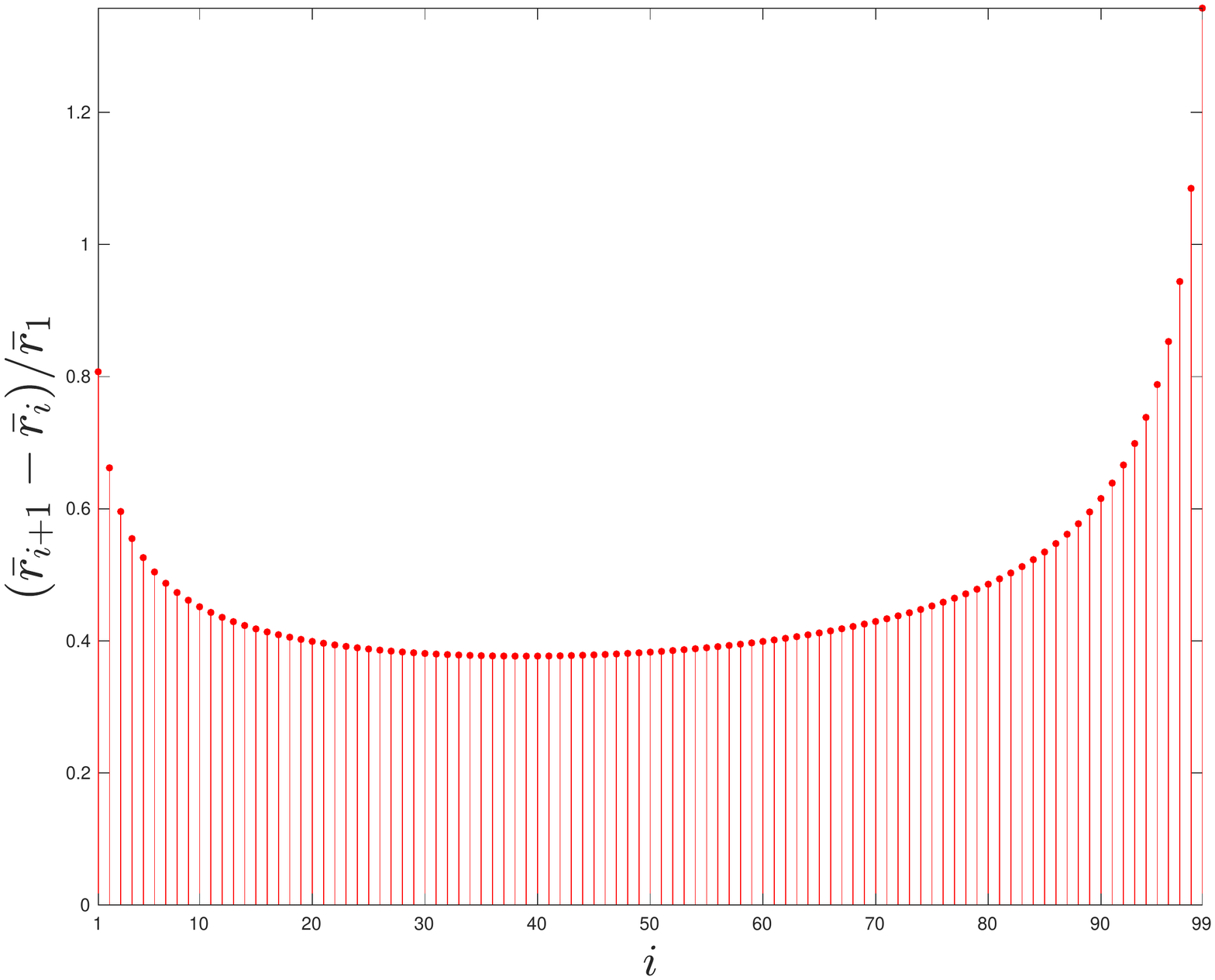}}\\
\subfigure[$\ell=100$]{
\label{fig:radii-c}
\includegraphics[height=2in]{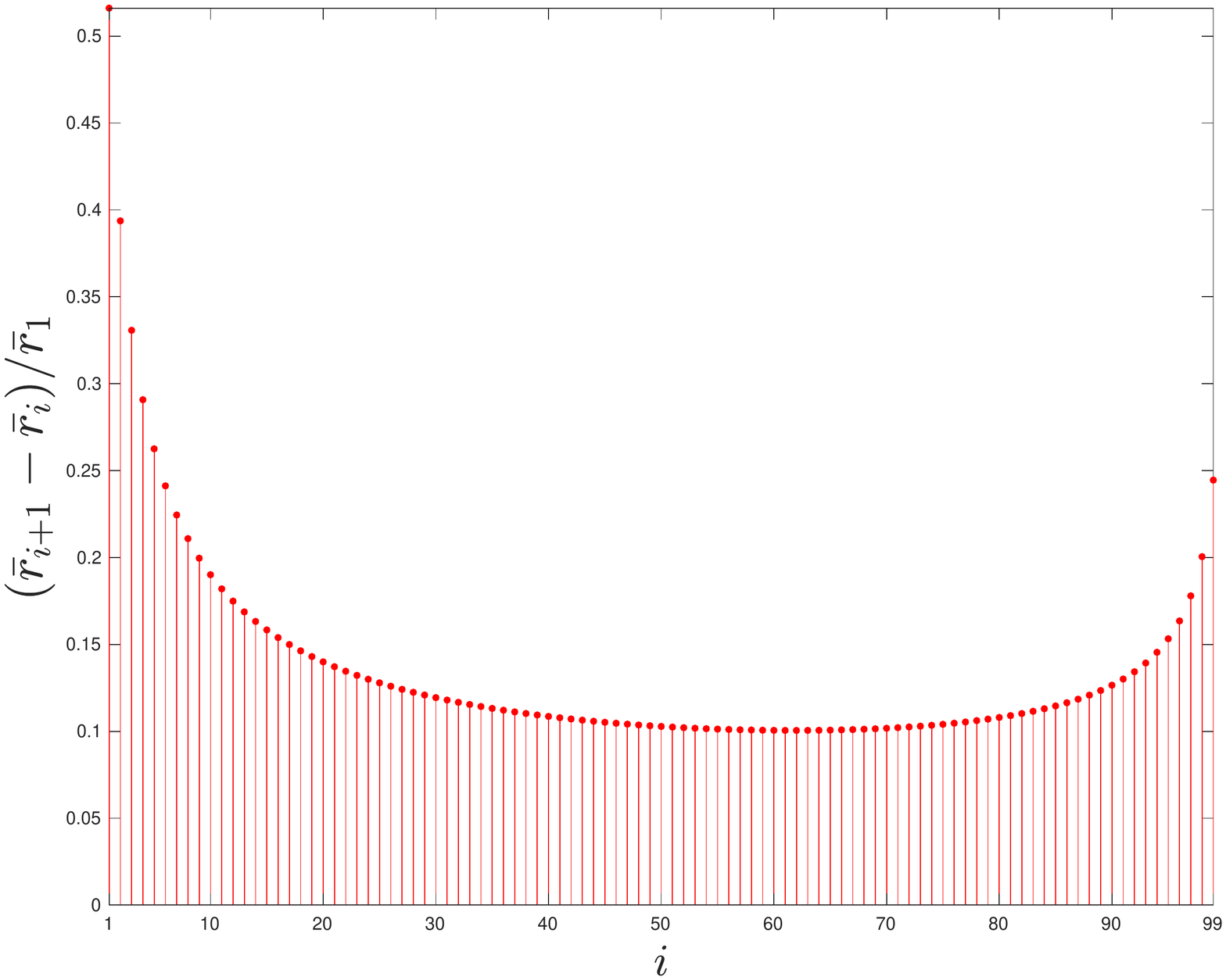}} 
\hspace{8pt}
\subfigure[$\ell=200$]{
\label{fig:radii-d}
\includegraphics[height=2in]{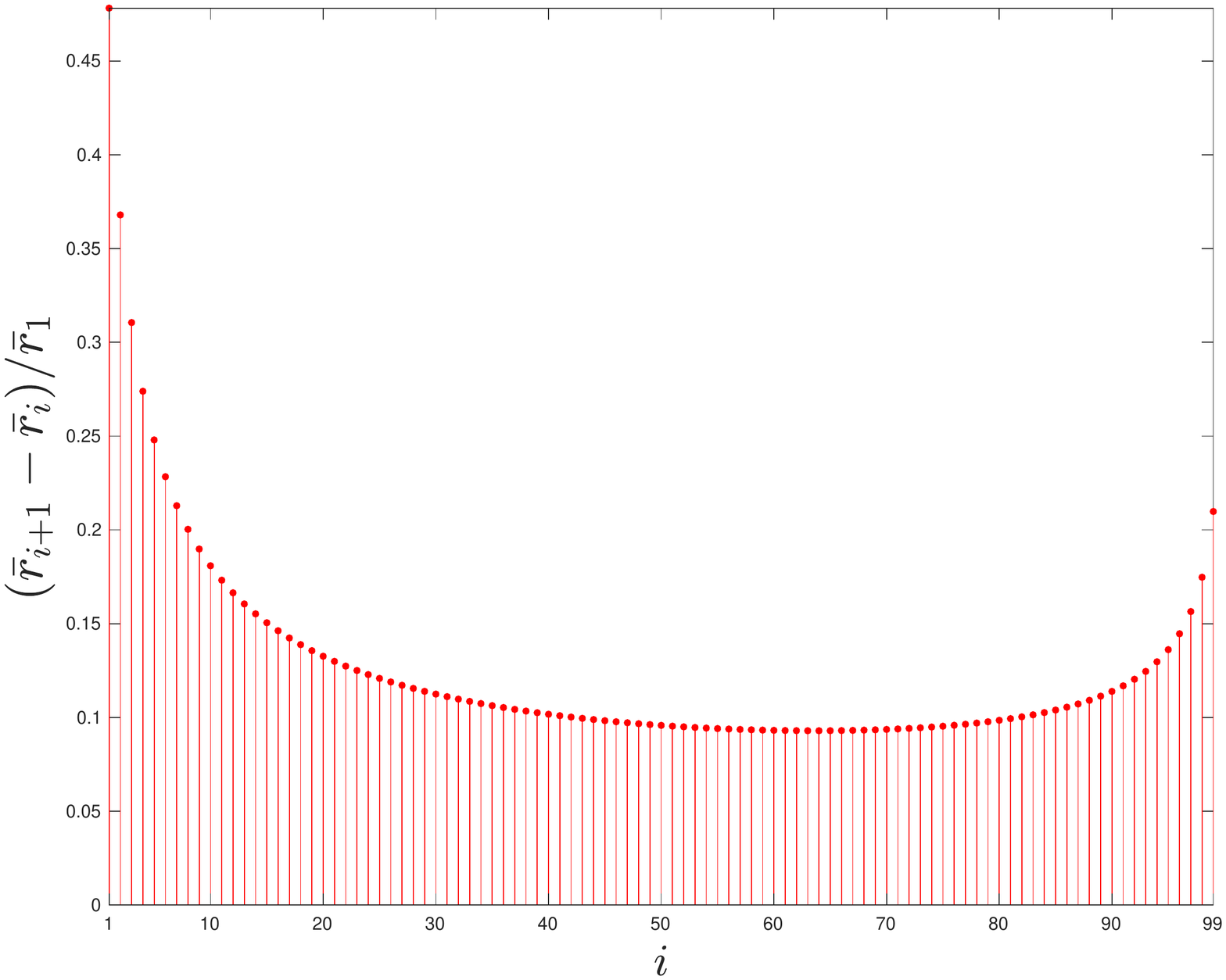}}
\caption{Spacing between consecutive circles of spiderweb central configurations with circles of equal mass and $\lambda=-1$, $n=100$, $m_0=0$ for different values of $\ell$.}\label{fig:radii}
\end{figure}

\begin{figure}
\centering
\includegraphics[width=3in]{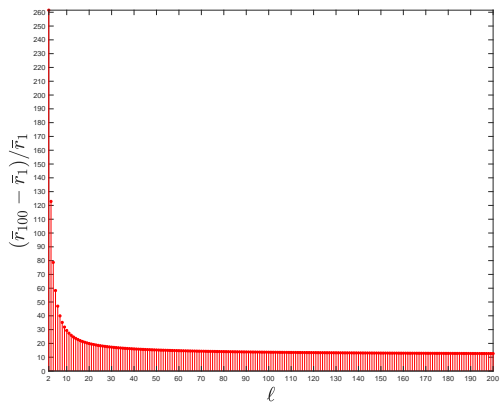}
\caption{Evolution of the relative width with respect to $\ell$ of spiderweb central configurations with circles of equal mass and $\lambda=-1$, $n=100$, $m_0=0$.}
\label{fig:width}
\end{figure}

\begin{conjecture} (See Figure~\ref{fig:radii}) 
For circles of equal mass and any $n\in \N$, $\ell \in \N_{\geq2}$, the sequence $\{a_i\}_{1\leq i\leq n-1}$ is convex. When $\ell=2$ and only in this case, the sequence is strictly increasing.

Moreover, let $a_{i^*}$ be the maximum of the $a_i$. From the convexity we know that $i^*\in\{1,n-1\}$ and, a priori,  we cannot exclude that $a_1= a_{n-1}$. Numerically however we have never seen $a_1= a_{n-1}$. Hence we could have a unique $i^*$. There exists an increasing function $\mu : \N \longrightarrow \N$ such that 

\[
i^* =
\begin{cases}
1, & \text{if }\quad \ell \geq \mu(n),\\
n-1, & \text{if }\quad \ell < \mu(n).
\end{cases}
\]

Numerically $\mu(n)$ seems small compared to $n$. Thus, $i^* = 1$ whenever $\ell \geq n$.
\end{conjecture}

Let $\chi(\eta) = \#\{ j\in \{1,\dots,n\} \,:\, r_j\leq \eta\}$. The mass distribution of a spiderweb central configuration, with $n$ circles, $m_0=0$ and $\ell$ equal masses per circle $m_1,\dots,m_n$, is given by

\[
M(\eta) = \ell \sum_{i=1}^{\chi(\eta)} m_i .
\]

By definition, $M(\eta)$ is given for the values of $\lambda$ and  $m_1,\dots,m_n$ chosen beforehand. What is remarkable is that while the sequence $\{a_1, \dots, a_n\}$ can have a very wild behavior when the sequence $\{m_1, \dots, m_n\}$ is irregular, the mass distribution $M(\eta)$ looks very regular (see Figures~\ref{figKappa} and \ref{fig:inv_kappa}). Indeed, to compensate for lighter masses on some circles, the neighbouring circles are closer.

This suggests that the general shape of the mass distribution is intrinsic to the spiderweb central configurations and deserves further study.

\begin{figure}
\centering
\subfigure[$\ell=2$]{
\label{fig:mass-a}
\includegraphics[height=2in]{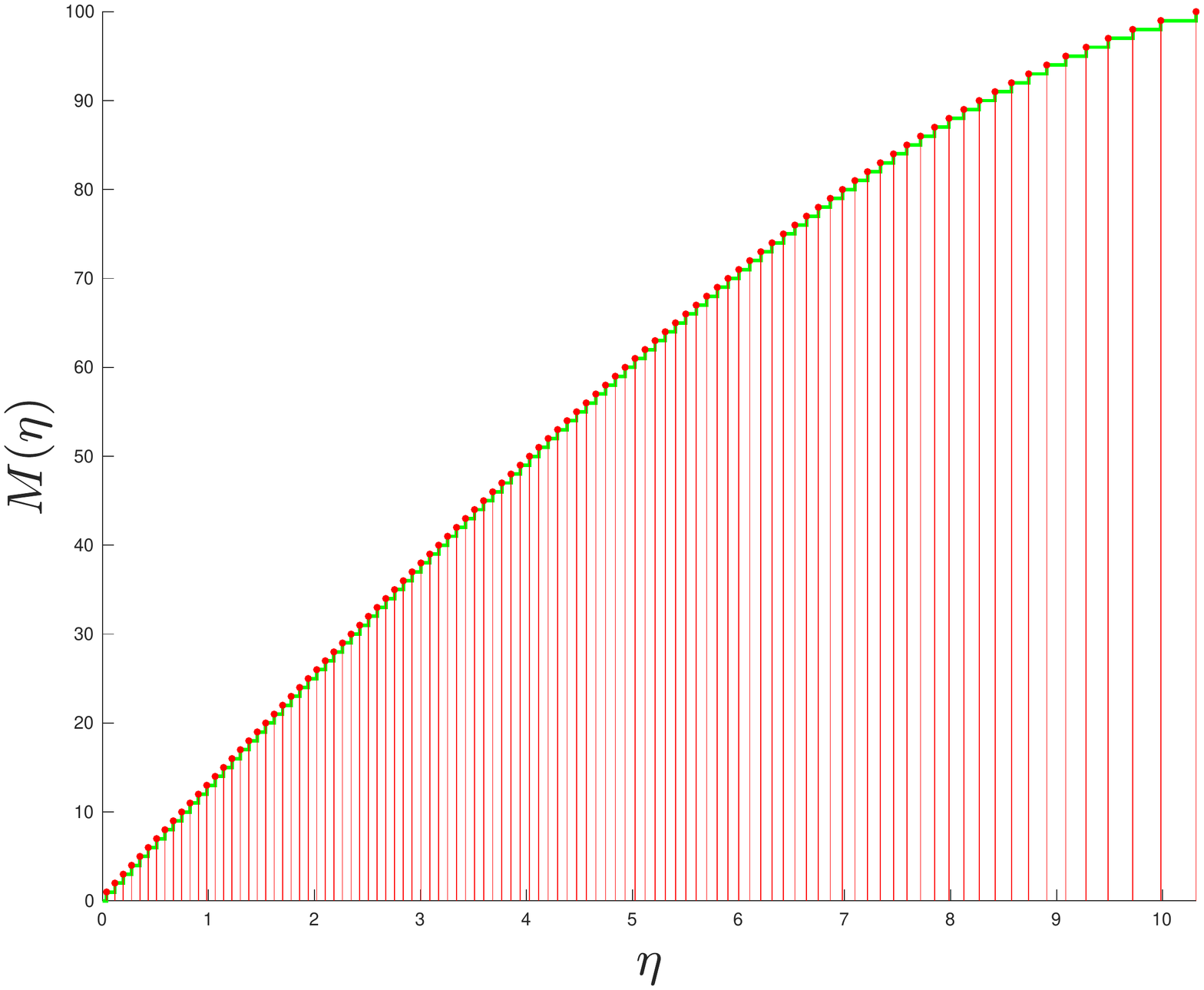}}
\hspace{8pt}
\subfigure[$\ell=6$]{
\label{fig:mass-b}
\includegraphics[height=2in]{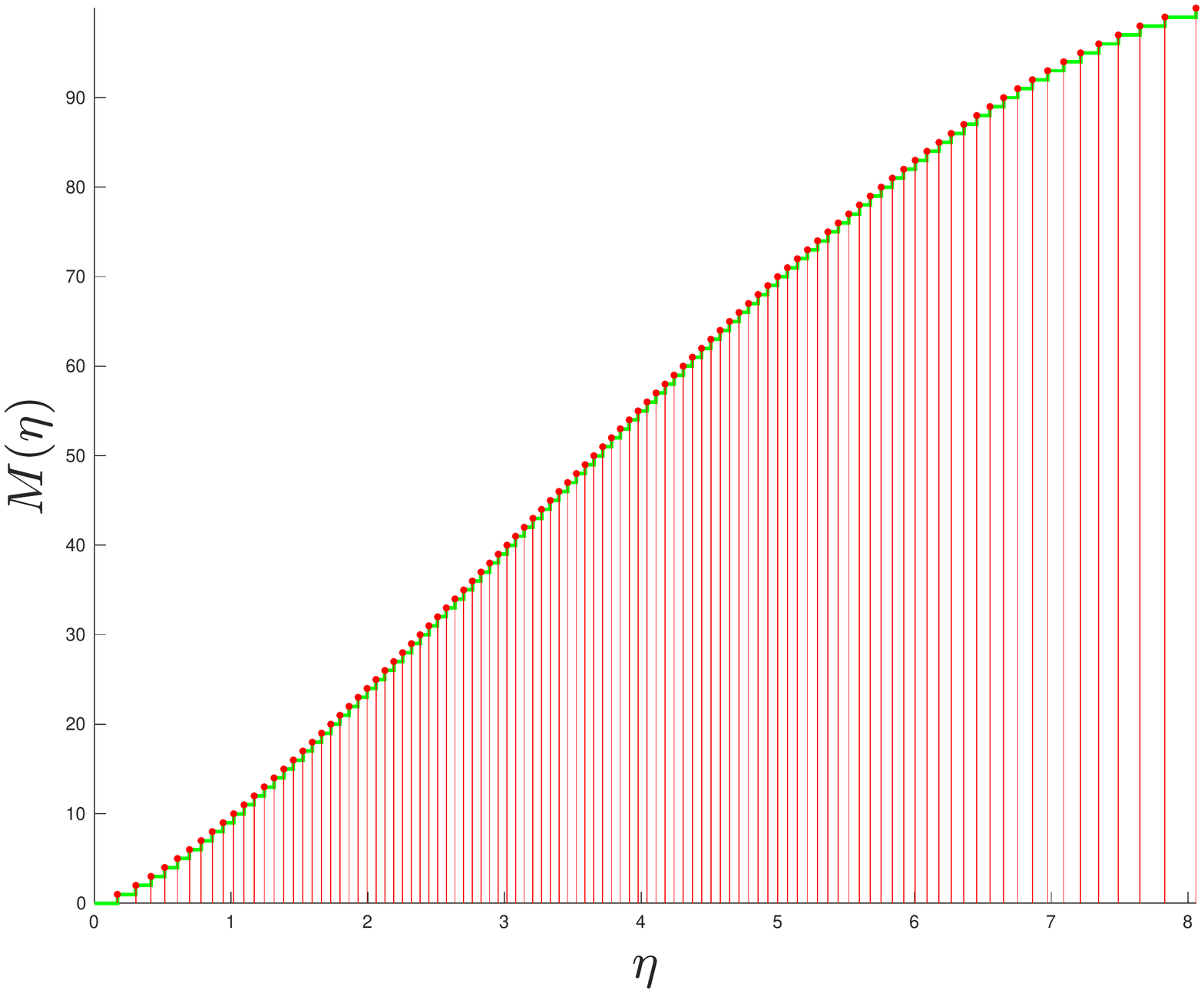}}\\
\subfigure[$\ell=100$]{
\label{fig:mass-c}
\includegraphics[height=2in]{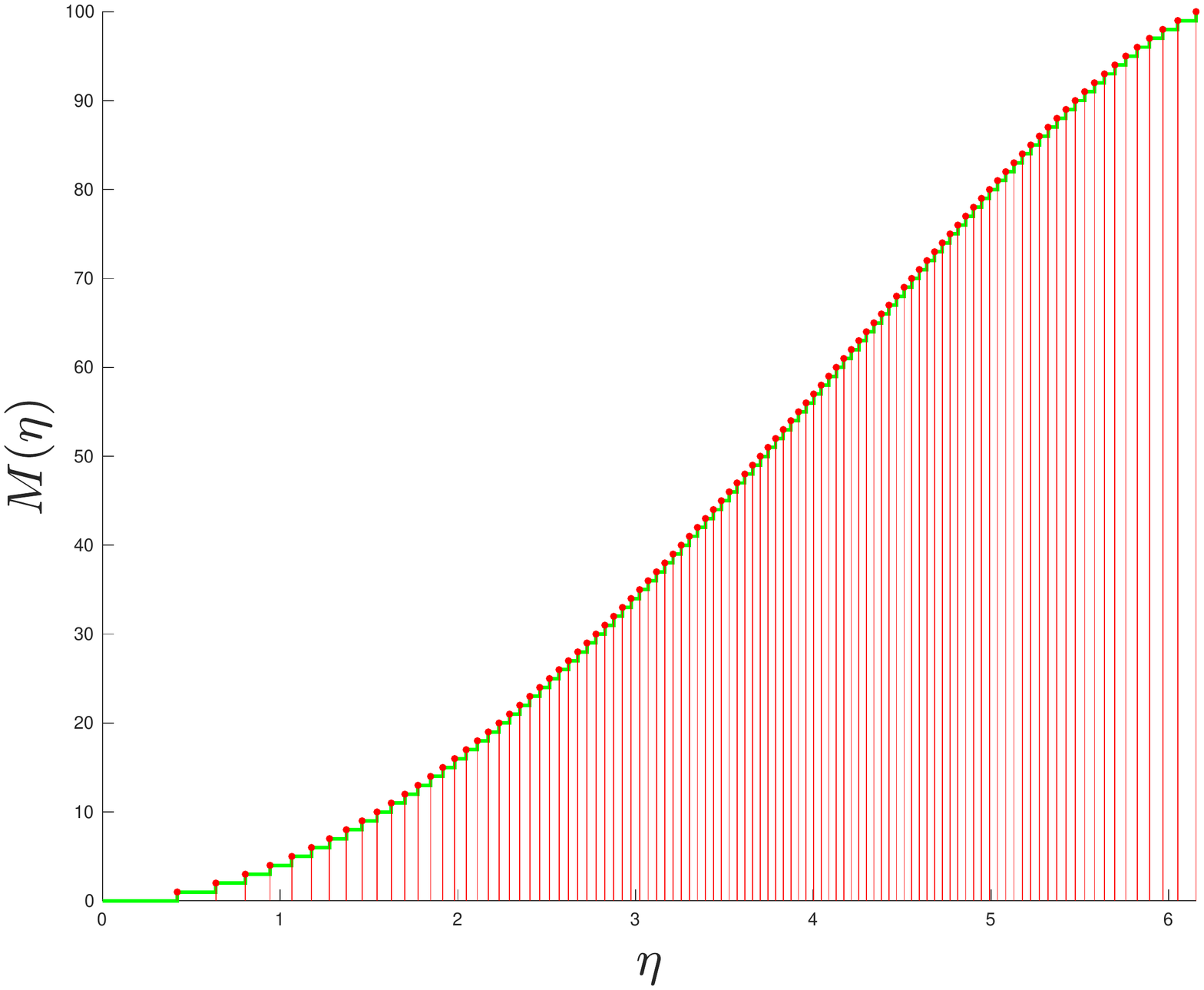}} 
\hspace{8pt}
\subfigure[$\ell=200$]{
\label{fig:mass-d}
\includegraphics[height=2in]{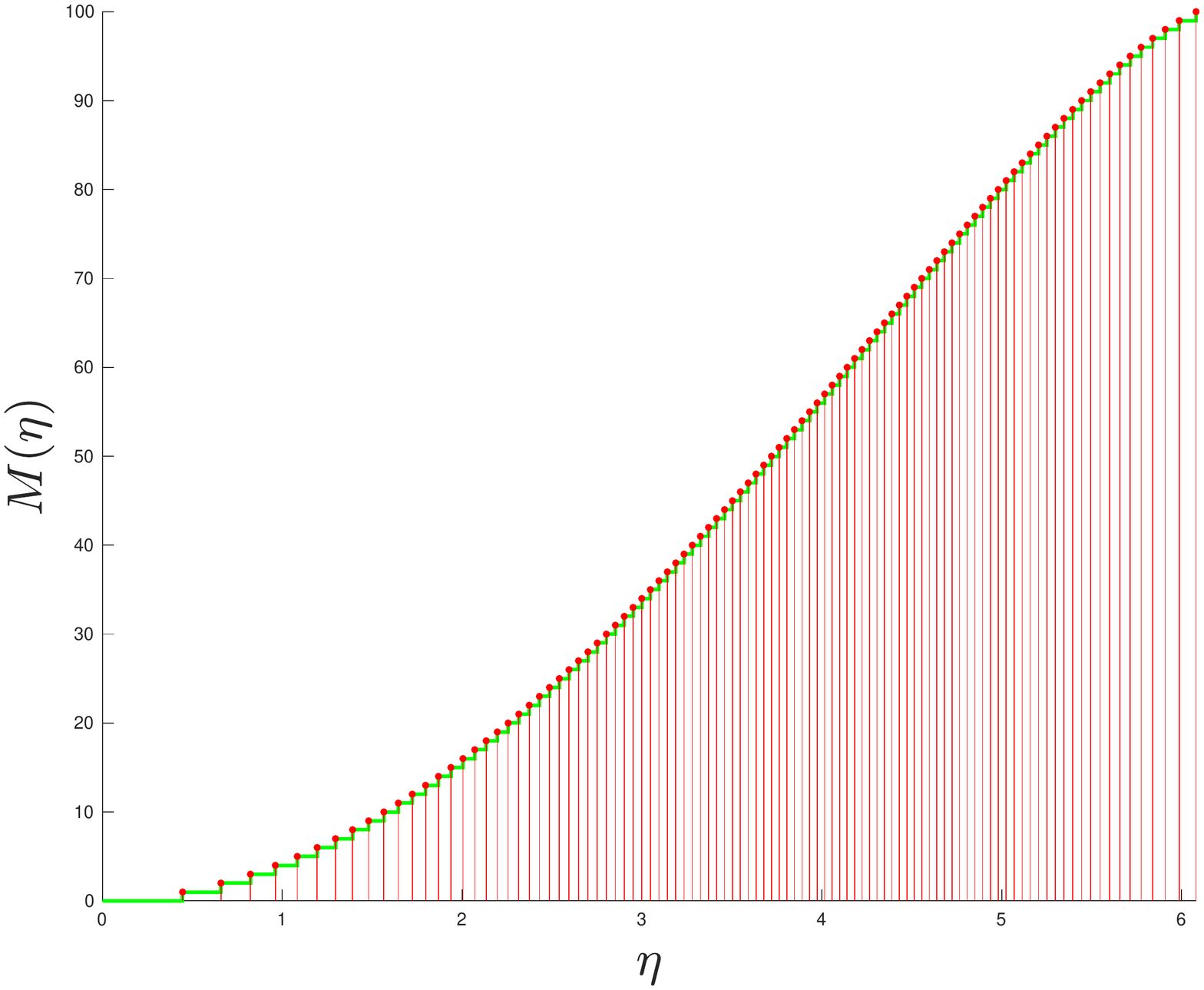}}
\caption{Mass distribution $M(\eta)$ of spiderweb central configurations with $\lambda=-1$, $n=100$, $(m_0,m)=(0,1/\ell,\dots,1/\ell)$ and different values of $\ell$.}
\label{fig:mass}
\end{figure}

\begin{figure}
\centering
\includegraphics[width = 3in]{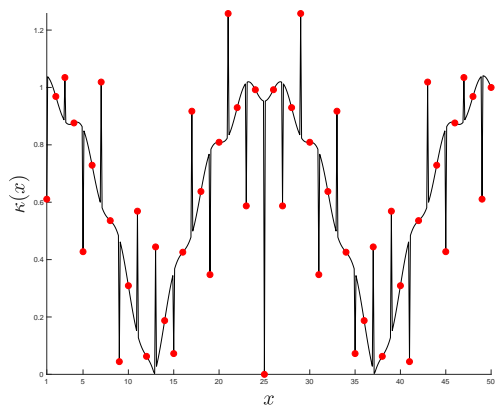}
\caption{Graph of $\kappa(x) = \left|\frac{\sin (21\pi(x-25))}{42\sin \left(\frac{\pi(x-25)}{2}\right)} + \cos\left(\frac{\pi x}{25}\right)\right|$ where $m_i=\kappa(i)$ for $i=1,\dots,50$ (\textcolor{red}{\textbullet}).}
\label{figKappa}
\end{figure}

\begin{figure}[ht]
\centering
\subfigure[Sequence $\{a_i\}$]{
\label{fig:radii-inv}
\includegraphics[height=2in]{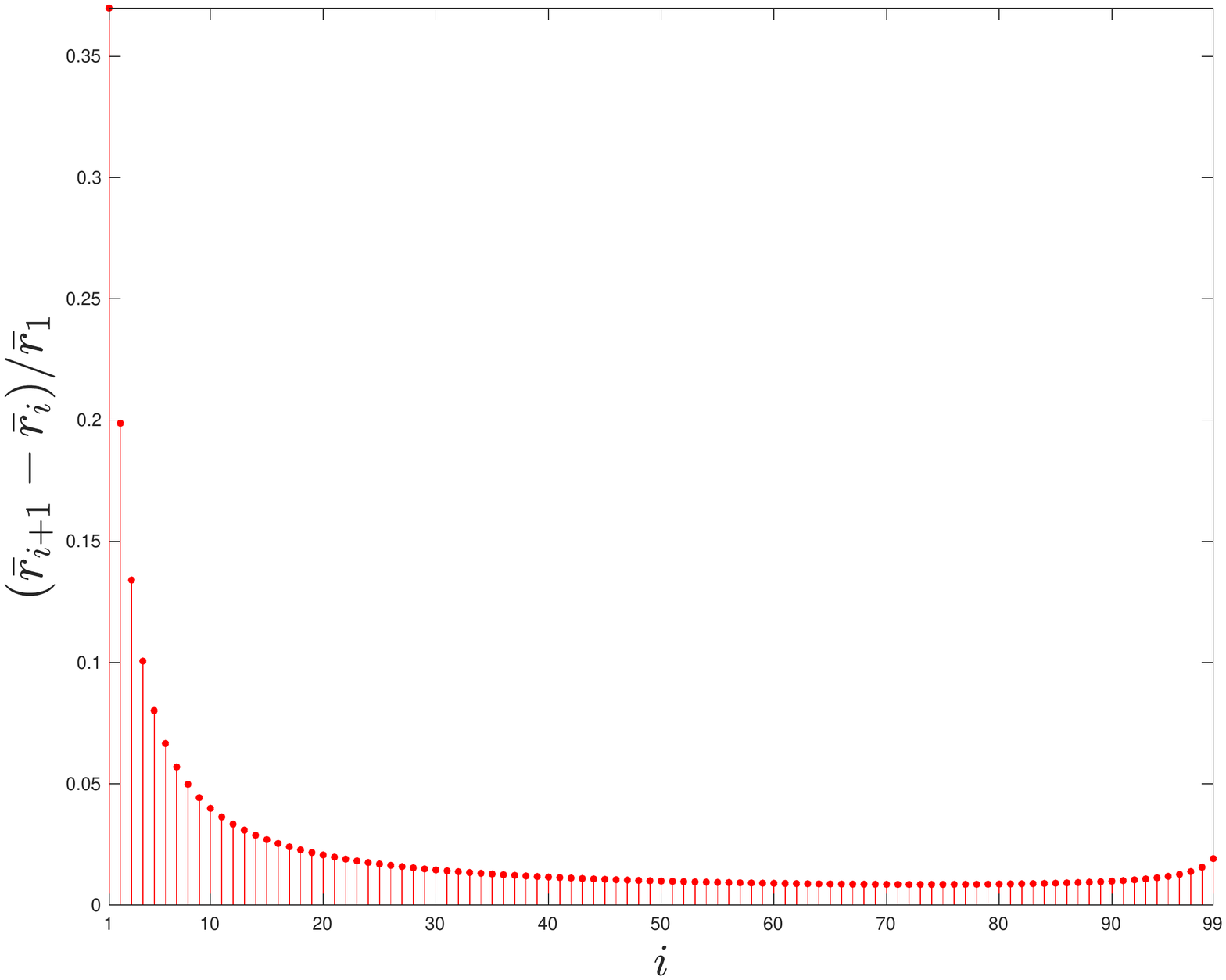}} 
\hspace{8pt}
\subfigure[Sequence $\{a_i\}$]{
\label{fig:radii-kappa}
\includegraphics[height=2in]{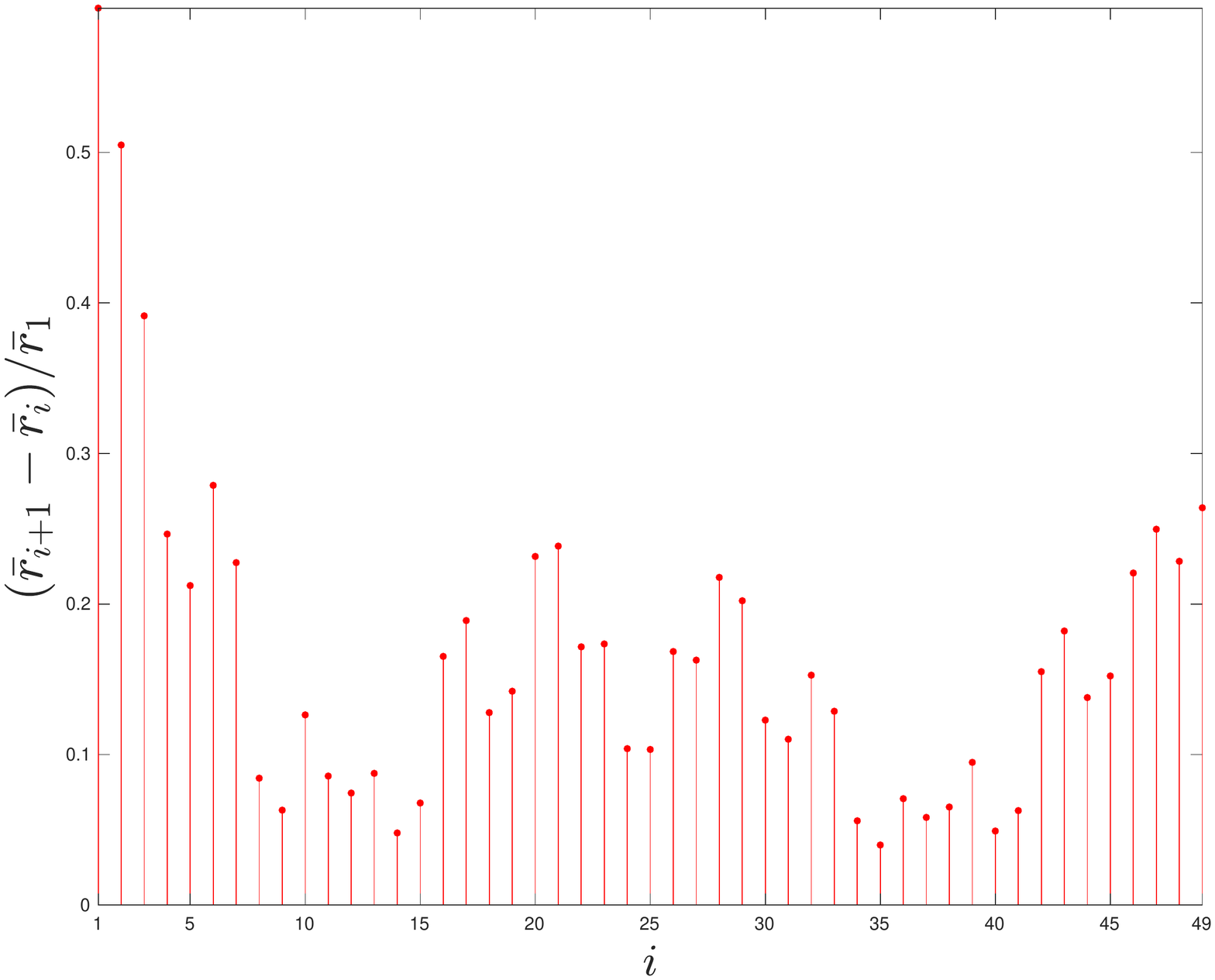}}\\
\subfigure[Corresponding $M(\eta)$]{
\label{fig:mass-inv}
\includegraphics[height=2in]{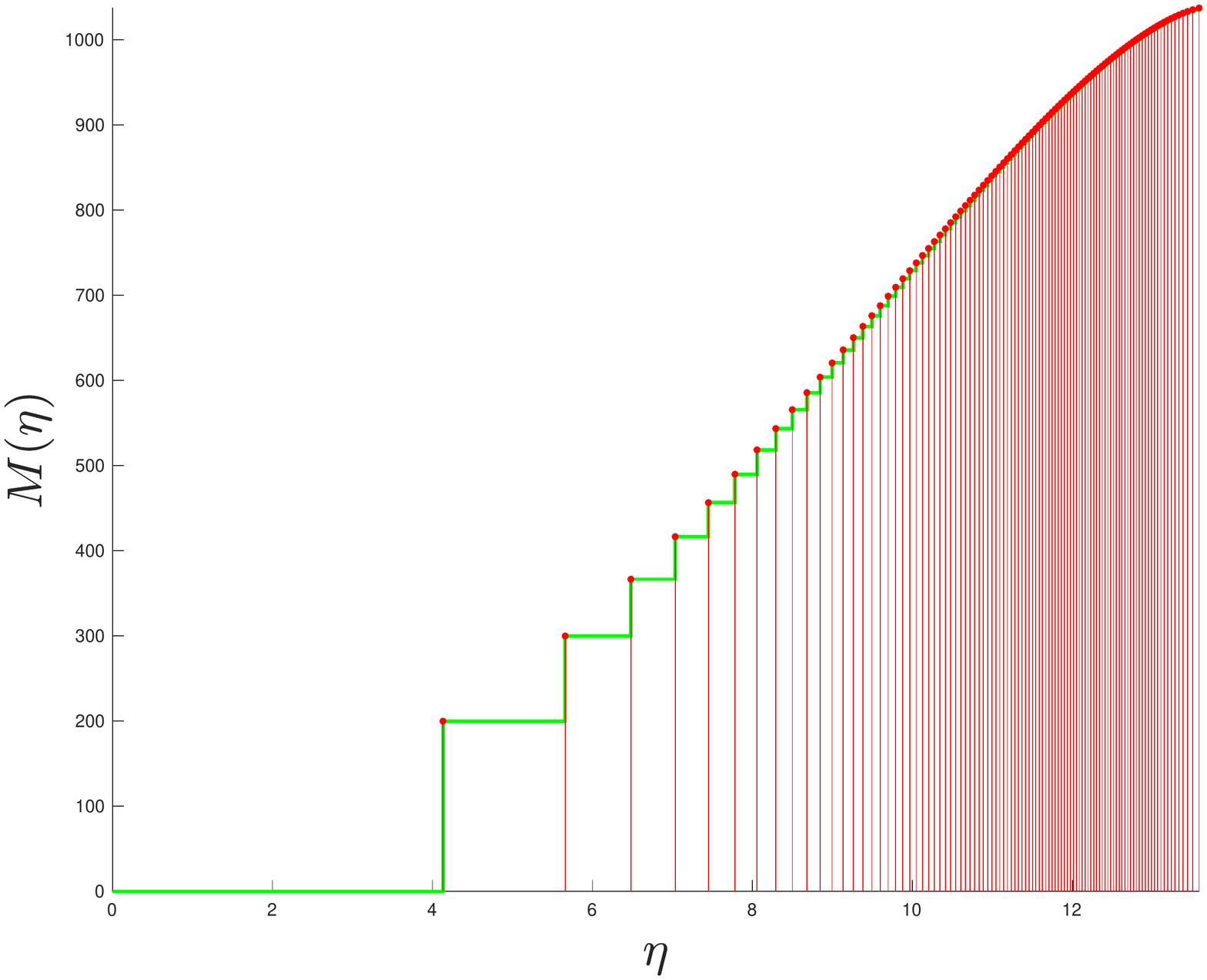}}
\hspace{8pt}
\subfigure[Corresponding $M(\eta)$]{
\label{fig:mass-kappa}
\includegraphics[height=2in]{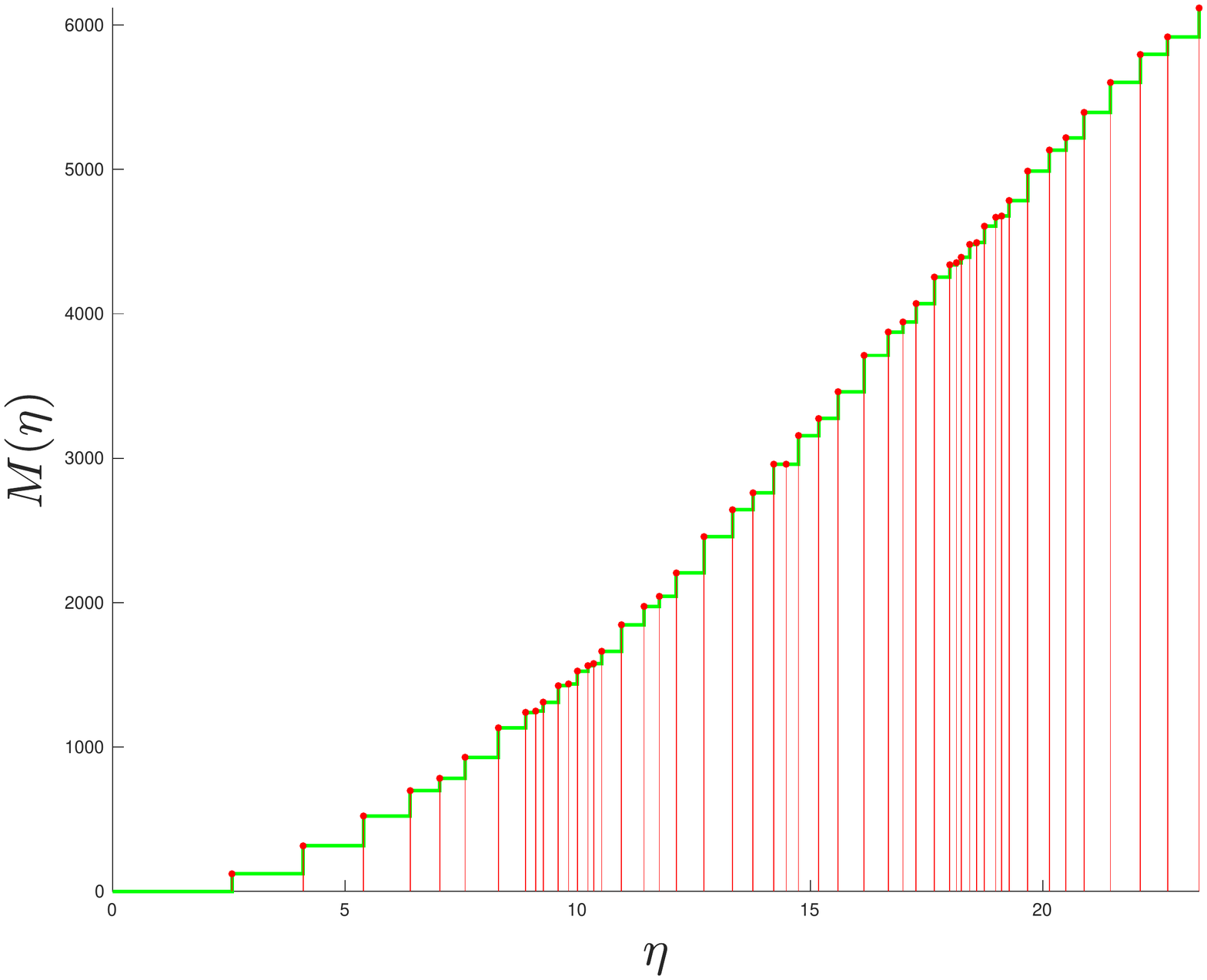}}
\caption{For $\lambda=-1$,
\subref{fig:radii-inv}-\subref{fig:radii-kappa} represent the spacing between consecutive circles for the spiderweb central configurations with $n=100$, $\ell=200$, $(m_0,m) = (0,1,1/2,\dots,1/100)$ and $n=50$, $\ell=200$, $(m_0,m) = (0,\kappa(1),\dots,\kappa(n))$ respectively, while
\subref{fig:mass-inv}-\subref{fig:mass-kappa} show their respective mass distribution. }
\label{fig:inv_kappa}
\end{figure}

\section{Acknowledgements} We are grateful to Jean-Philippe Lessard for helpful discussions and inputs on the computational part of the paper.

\end{document}